\documentclass[onefignum,onetabnum]{siamart171218}

\usepackage{lipsum}
\usepackage{amsfonts}
\usepackage{graphicx}
\usepackage{epstopdf}
\usepackage{hyperref}
\usepackage{algorithmic}
\ifpdf
    \DeclareGraphicsExtensions{.eps,.pdf,.png,.jpg}
\else
    \DeclareGraphicsExtensions{.eps}
\fi
\usepackage{amsopn}
\usepackage{todonotes}
\usepackage{stmaryrd}
\usepackage{mathabx}
\allowdisplaybreaks[4]
\usepackage{physics}
\usepackage{xifthen}
\usepackage{comment}
\usepackage{multirow}
\usepackage{booktabs}

\newsiamremark{remark}{Remark}

\DeclareMathOperator*{\argmin}{arg\,min}

\DeclareMathOperator{\spanfun}{span}
\newcommand{\fnorm}[1]{\norm{#1}_{\text{F}}}
\newcommand{\dimcont}[2]{{#1}_1 \times \cdots \times {#1}_{#2}}
\newcommand{\idxcont}[2]{#1_1, \dots, #1_{#2}}

\newcommand{\mhp}{multi-homogeneous-poly}

\newcommand{\tX}{\mathbf{X}}
\newcommand{\tY}{\mathbf{Y}}
\newcommand{\tT}{\mathbf{T}}

\newcommand{\tV}{\mathbf{V}}
\newcommand{\tu}[1]{\ifthenelse{\isempty{#1}}{\vec{\mathbf{u}}}
{\mathbf{u}^{[#1]}}}
\newcommand{\tx}[1]{\ifthenelse{\isempty{#1}}{\vec{\mathbf{x}}}
{\mathbf{x}^{[#1]}}}
\newcommand{\tv}[1]{\ifthenelse{\isempty{#1}}{\vec{\mathbf{v}}}
{\mathbf{v}^{[#1]}}}
\newcommand{\tw}[1]{\ifthenelse{\isempty{#1}}{\vec{\mathbf{w}}}
{\mathbf{w}^{[#1]}}}

\newcommand{\ocalU}{\mathcal{U}}
\newcommand{\ocalW}{\mathcal{W}}

\newcommand{\calJ}{\mathcal{J}}
\newcommand{\calM}{\mathcal{M}}
\newcommand{\calR}{\mathcal{R}}

\newcommand{\bbN}{\mathbb{N}}
\newcommand{\bbP}{\mathbb{P}}
\newcommand{\bbR}{\mathbb{R}}
\newcommand{\bbX}{\mathbb{X}}

\title{Tensor ring decomposition: optimization landscape and one-loop
convergence of alternating least squares}

\headers{Tensor Ring Decomposition}{Ziang Chen, Yingzhou Li,
and Jianfeng Lu}

\author{Ziang Chen
\thanks{School of Mathematical Sciences, Peking University
(\email{chenziang19970214@pku.edu.cn}).}
\and Yingzhou Li
\thanks{Department of Mathematics, Duke University 
(\email{yingzhou.li@duke.edu}).}
\and Jianfeng Lu
\thanks{Department of Mathematics, Department of Chemistry, and
Department of Physics, Duke University 
(\email{jianfeng@math.duke.edu}).}
}

\begin{document}

\maketitle	

\begin{abstract}
In this work, we study the tensor ring decomposition and its
associated numerical algorithms. We establish a sharp transition
of algorithmic difficulty of the optimization problem as the bond
dimension increases: On one hand, we show the existence of spurious
local minima for the optimization landscape even when the tensor ring
format is much over-parameterized, \textit{i.e.}, with bond dimension
much larger than that of the true target tensor.  On the other hand,
when the bond dimension is further increased, we establish one-loop
convergence for alternating least square algorithm for the tensor ring
decomposition. The theoretical results are complemented by numerical
experiments for both local minima and the one-loop convergence for
the alternating least square algorithm.
\end{abstract}



\section{Introduction}

Tensors are ubiquitous especially for representing high dimensional
functions or operators.  However, due to the curse of dimensionality,
both the storage cost and the computational cost of vanilla tensors
scale exponentially as the dimension increases.

To overcome the curse of dimensionality, tensor network decomposition
has been widely used, in particular in the physics community,
to represent specific families of high dimensional tensors, with
much fewer degrees of freedom than an arbitrary tensor. Perhaps
the most famous tensor network is the \emph{matrix product state}
(MPS)~\cite{Affleck1988, Schollwock2011}. The matrix product
state forms the basis of the density matrix renormalization
group algorithm~\cite{White1992, White1993}, which has been
widely used in physics and chemistry literature and extremely
successful for one-dimensional many-body physical systems, see
\textit{e.g.}, reviews~\cite{Chan2011, Hallberg2006, Schollwock2005,
WoutersVanNeck2014}. In fact, it has been shown rigorously that the
ground state for a gapped one-dimensional many-body physical system can
be efficiently represented by a matrix product state~\cite{Arad2017,
Brandao2015, Hastings2007}.  To go beyond and deal with more
general physical systems, the matrix product state has also been
extended to other tensor network decomposition formats, including
PEPS~\cite{Verstraete2004a}, MERA~\cite{Vidal2007}, just to name a few.
See~\cite{Orus2014} for a recent review on tensor networks.

In the mathematics literature, the matrix product state is
known as the \emph{tensor train} (TT) format~\cite{Oseledets2011,
Oseledets2009}, which is a special case of the hierarchical Tucker
format~\cite{Grasedyck2010, Hackbusch2009}. The study of tensor
networks, including their algebraic and geometric structures and
associated algorithms, have received increasing attention from the
mathematics community as well, see \textit{e.g.},~\cite{Espig2015,
Landsberg2012, RobevaSeigal, Rohwedder2013, Ye2018}. On the
algorithmic side, most of the existing works focus on the tensor train
format~\cite{Bigoni2016, Holtz2012, Oseledets2011}. The convergence
analysis for the construction and the compression algorithms of TT
is established in~\cite{Espig2015, Rohwedder2013}.

The \emph{tensor ring} (TR) format extends TT format
to accommodate periodic boundary conditions, hence, is
suggested as the ansatz for periodic one-dimensional physical
systems~\cite{Verstraete2004}. Unfortunately, the construction and
the compression of TR turn out to be much more difficult than that
of TT. Most algorithms working efficiently for TTs cannot be easily
extended to TRs. \emph{Alternating least square} (ALS) algorithm is
one exception, but it still relies on carefully designed sampling
techniques and initial guesses~\cite{Khoo2017a} to be efficient.
In addition to the difficulty in designing efficient algorithms, the
inability to compute the exact minimal bond dimension TR decomposition
is numerically demonstrated recently~\cite{Batselier2018}. Although
some success of TR is achieved in compressing tensors in
practice~\cite{Khoo2017a, Wang2017c, Zhao2018, Zhao2016}, TR
decomposition and TR operations remain a challenging problem.
The mathematical understanding of TR format and associated algorithms
is still rather sparse. Motivated by such a gap, in this work,
we analyze the tensor ring decomposition. While TR is arguably the
simplest tensor network beyond the TT format, our study may also shed
some light on more complicated tensor network formats.

Let us also mention that there are other tensor decomposition
formats besides tensor networks in the literature, such as tensor
rank decomposition (often known as the CANDECOMP/PARAFAC (CP)
decomposition) and Tucker decomposition. Many works have been
devoted in designing efficient algorithms for finding near-optimal
CP and Tucker decomposition (see review article~\cite{Kolda2009}
and references therein).

\subsection{Contribution}

In this work, we analyze the optimization landscape of the TR
decomposition and prove the existence of spurious local minima even
if the TR format is over-parameterized. More precisely, we propose
a particular $d$-th order tensor as the target tensor, which is of
TR format with bond dimension $r+1$, and a spurious local minimum
is identified in the space of TR with bond dimension $r^{d-1}$. Note
that the bond dimension scales exponentially with $d$. Such a spurious
local minimum casts trouble for optimization problem associated with
TR decomposition. Although the spurious local minimum identified might
not be strict, we numerically validate that ALS in some sense cannot
escape from a small neighborhood of the spurious local minimum.

Our second result establishes the one-loop convergence of ALS algorithm
for TR decomposition if we even further lift the bond dimension of
the proposed TR space. More precisely, for any target $d$-th order
tensor of TR format with bond dimension $r$ satisfying some full-rank
conditions and starting from a random initial TR, ALS almost surely
converges to the target tensor after one loop iteration, when the
bond dimension of the proposed TR space is $r^{d-1}$.

Combining two results together, we establish a sharp transition between
the triviality of ALS, \textit{i.e.}, the one-loop convergence,
and the existence of spurious local minima. Up to some subtle
differences, the results shown here for TR are similar to that for
TT in~\cite{Rohwedder2013}. Investigation of any of these subtle
differences leads to the intrinsic difference between TT and TR,
\textit{i.e.}, TT of fixed bond dimension is a closed set whereas TR
of fixed bond dimension is not a closed set~\cite{Landsberg2012}.

\subsection{Organization}

The rest of this paper is organized as follows. In Section~\ref{sec:
Preliminaries}, we introduce some basics for tensor ring format. In
Section~\ref{sec: spurious local minima}, we introduce tensor ring
decomposition and analyze the optimization landscape. The existence
of spurious local minima is established. In Section~\ref{sec: ALS},
we introduce the alternating least square algorithm for tensor ring
decomposition and show the one-loop convergence. In Section~\ref{sec:
numerical}, we give some numerical validations of the theoretical
results. The paper is concluded in Section~\ref{sec: conclusion}.

\section{Preliminaries}
\label{sec: Preliminaries}

In this section, we first introduce some tensor notations that will
be used throughout the paper and then provide the precise definition
of the tensor ring format.

\subsection{Tensor notations}

While tensor is a powerful tool in many areas, the corresponding
notations are somewhat complicated. In this section, we introduce
a few common notations in representing a tensor. A $d$-th order
tensor $\tX$ is a $d$-dimensional array, \textit{i.e.}, $\tX
\in \bbR^{\dimcont{n}{d}}$, where $\vec{n} = (\idxcont{n}{d})
\in \bbN_+^d$ is the size of the tensor. $\vec{n}$ is also called
the external dimension of $\tX$.  Entries of $\tX$ are denoted
as $\tX(\idxcont{x}{d})$, where $1 \leq x_i \leq n_i$ denotes the
$i$-th index of the tensor for $i = 1, 2, \dots, d$.  Matlab colon
(:) notation is powerful in representing contiguous entries of a
tensor. For example, let $\tX \in \bbR^{n_1 \times n_2 \times n_3}$
be a 3-rd order tensor: $\tX(:, x_2, x_3)$ denotes a vector in
$\bbR^{n_1}$ and $\tX(:, x_2, :)$ denotes a $n_1 \times n_3$ matrix.

For two tensors of the same size
$\mathbf{X},\mathbf{Y}\in\mathbb{R}^{n_1\times n_2\times\cdots\times
n_d}$, the inner product between $\mathbf{X}$ and $\mathbf{Y}$ is
defined via
\begin{equation}
    \langle \tX,\tY\rangle=\sum_{x_1=1}^{n_1} \sum_{x_2=1}^{n_2} \cdots
    \sum_{x_d=1}^{n_d}  \tX(\idxcont{x}{d})\tY(\idxcont{x}{d}).
\end{equation}
The Frobenius norm is used as a distance measure of $\mathbf{X}$,
which is given as
\begin{equation} \label{eq:fnorm}
    \fnorm{\tX}=\langle \tX,\tX\rangle^{1/2} = \left( \sum_{x_1=1}^{n_1}
    \sum_{x_2=1}^{n_2} \cdots \sum_{x_d=1}^{n_d} \tX(\idxcont{x}{d})^2
    \right)^{1/2}.
\end{equation}

Suppose that $\tX \in \bbR^{\dimcont{n}{d_1}}$ and $\tY \in
\bbR^{\dimcont{m}{d_2}}$ are two tensors. The tensor product of $\tX$
and $\tY$ is $\tX \otimes \tY \in \bbR^{\dimcont{n}{d_1} \times
\dimcont{m}{d_2}}$, with entries given by
\begin{equation}
    (\tX \otimes \tY)(\idxcont{x}{d_1}, \idxcont{y}{d_2}) =
    \tX(\idxcont{x}{d_1}) \tY(\idxcont{y}{d_2}),
\end{equation}
where $x_i = 1, 2, \dots, n_i$, $i = 1, 2, \dots, d_1$ and
$y_j = 1, 2, \dots, m_j$, $j = 1, 2, \dots, d_2$.

\subsection{Tensor ring} 

A $d$-th order tensor ring is a periodic consecutive product of $d$
$3$-rd order tensors, which can be viewed as a periodic version of a
$d$-th order tensor train. Before rigorously defining tensor ring,
we first recall a periodic index $i$ of periodicity $d$ as $i =
\mathrm{mod}(i-1, d) + 1$. Hence we have $i \in \{1, 2, \dots,
d\}$ for any integer $i$. A sequence $\vec{t} = (\idxcont{t}{d})$
is periodically indexed if $t_i$ for $i < 1$ or $i > d$ is allowed
and $t_i \equiv t_{\mathrm{mod}(i-1,d)+1}$, e.g., $t_0 \equiv t_d$,
$t_{d+1} \equiv t_1$, etc. These periodic indices and periodically
indexed sequences significantly reduce the redundancy of notations
in tensor ring. Define
\begin{equation}
    \ocalU_{\vec{r},\vec{n}}^d = \bigtimes_{i=1}^d \bbR^{r_i
    \times n_i \times r_{i+1}},
\end{equation}
where $\vec{n} = (\idxcont{n}{d})$ and $\vec{r} = (\idxcont{r}{d})$
denote the external and internal dimension respectively. An element
in $\ocalU_{\vec{r},\vec{n}}^d$ is denoted as $\tu{} = \left(\tu{1},
\tu{2}, \dots, \tu{d}\right) \in \ocalU_{\vec{r},\vec{n}}^d$, where
$\tu{i} \in \bbR^{r_i \times n_i \times r_{i+1}}$ for $i = 1, 2,
\dots, d$. In order to shorten notations, we use two notations for
two unfolding tensors of $\tu{i}$, \textit{i.e.}, $\tu{i}_{k_1,k_2} =
\tu{i}(k_1,:,k_2)$ is a vector in $\mathbb{R}^{n_i}$ and $\tu{i}(x_i)
= \tu{i}(:,x_i,:)$ is a $r_{i} \times r_{i+1}$ matrix.

Let $\tau$ denote the mapping from $\ocalU_{\vec{r},\vec{n}}^d$
to a $d$-th order tensor of external dimension $\vec{n}$ as {}
\begin{equation}
    \tau(\tu{}) = \sum_{\idxcont{k}{d}} \tu{1}_{k_1,k_2} \otimes
    \tu{2}_{k_2,k_3} \otimes \cdots \otimes \tu{d}_{k_d,k_1}, \quad
    \tu{} \in \ocalU_{\vec{r},\vec{n}}^d.
\end{equation}
Throughout this paper, we abuse the mapping notation $\tau$ for any
external and internal dimension $\vec{n}$ and $\vec{r}$ and order $d$.
Elements of $\tau(\tu{})$ can be evaluated as
\begin{equation}
    \tau(\tu{})(\idxcont{x}{d}) = \tr \left( \tu{1}(x_1) \tu{2}(x_2)
    \cdots \tu{d}(x_d) \right),
\end{equation}
for $x_i = 1, 2, \dots, n_i$, $i = 1, 2, \dots, d$.

A $d$-th order tensor $\tT$ of external dimension $\vec{n}$ has
tensor ring format of internal dimension $\vec{r}$ if there exists a
$\tu{} \in \ocalU_{\vec{r},\vec{n}}^d$ such that $\tT = \tau(\tu{})$.
We denote the collection of such tensors as
\begin{equation}
    \calR_{\vec{r},\vec{n}}^d = \left\{ \tT \mid \tT = \tau(\tu{}),
    \quad \tu{} \in \ocalU_{\vec{r},\vec{n}}^d \right\},
\end{equation}
which includes all TRs with rank bounded by $r$.  We call $\max_{1\leq
i\leq d} r_i$ the \emph{bond dimension} of the TR format. TR format
is a special case of tensor networks, where the underlying network is
a one-dimensional ring. More detailed discussions on tensor networks
can be found in \textit{e.g.}, \cite{Orus2014}.

Each element in $\ocalU_{\vec{r},\vec{n}}^d$ corresponds to one
tensor in $\calR_{\vec{r},\vec{n}}^d$. However, the reverse is not
true. Actually each element in $\calR_{\vec{r},\vec{n}}^d$ corresponds
to infinitely many elements in $\ocalU_{\vec{r},\vec{n}}^d$.  Let
$\tu{}$ be an element in $\ocalU_{\vec{r},\vec{n}}^d$ corresponding
to $\tT \in \calR_{\vec{r},\vec{n}}^d$.  Given any tuple of $d$
invertible matrices, $\vec{A} = \left( \idxcont{A}{d} \right)$, where
$A_i \in \mathrm{GL}(r_i, \bbR)$ is a $r_i\times r_i$ invertible
matrix for $i=1,2,\cdots,d$, we define $\theta_{\vec{A}}(\tu{}) =
\tv{} = \left(\tv{1}, \tv{2}, \dots, \tv{d} \right) $ via
\begin{equation}
    \tv{i}(x_i) = A_i \tu{i}(x_i) A_{i+1}^{-1}, \quad x_i = 1, 2,
    \dots, n_i, \text{ and } i = 1, 2, \dots, d,
\end{equation}
where periodic indexing is applied, i.e., $A_{d+1} = A_1$. It can
be easily seen that $\tau(\theta_{\vec{A}}(\tu{})) = \tau(\tu{})$,
\textit{i.e.}, $\theta_{\vec{A}}(\tu{})$ and $\tu{}$ correspond to
the same $d$-th order tensor. This is called \emph{gauge freedom}
or \emph{gauge invariance}. The orbit
\begin{equation}
    \calM_{\tu{}} = \left\{ \theta_{\vec{A}}(\tu{}) \mid \vec{A} \in
    \bigtimes_{i=1}^d \mathrm{GL}(r_i,\bbR) \right\}
\end{equation}
is called the manifold due to the gauge freedom.

\section{Tensor ring decomposition and spurious local minima}
\label{sec: spurious local minima}

Tensor ring provides an efficient representation of high-order
tensors, especially for those tensors with underlying physical
geometry being a one-dimensional ring, and thus the periodicity becomes
natural. However, finding such a tensor ring representation of a given
high-order tensor is highly nontrivial in practice. In this section, we
first cast the tensor ring decomposition as a constrained optimization
problem, which is widely used in the literature~\cite{Khoo2017a,
Zhao2016}. Then an explicit spurious local minimum is constructed for
the relaxed version of the constrained optimization. Such a spurious
local minimum to some degree explains why tensor ring decomposition
is much more difficult than tensor train decomposition and matrix
factorization in practice. Recall that tensor train decomposition
has one-loop convergence for exactly parameterized constraint
set~\cite{Oseledets2011, Rohwedder2013} (the one-loop convergence for
tensor ring will be discussed in Section~\ref{sec: one-loop}); and it
is well known that matrix factorization including matrix eigenvalue
decomposition~\cite{Li2019} and low-rank factorization~\cite{Ge2017}
does not have spurious local minima even in the under-parameterized
regime.

Let $\tT \in \bbR^{\dimcont{n}{d}}$ be the target $d$-th order
tensor.  \emph{Tensor ring decomposition} aims to find a $\tu{}
\in \ocalU_{\vec{r},\vec{n}}^d$ such that the distance between $\tT$
and $\tau(\tu{})$ is minimized. If the tensor Frobenius norm is used
as the distance, we can formulate the tensor ring decomposition as
the following constrained optimization problem:
\begin{equation} \label{eq:trdobj}
    \min_{\tu{}\, \in\, \ocalU_{\vec{r},\vec{n}}^d} \frac{1}{2}
    \fnorm{ \tT - \tau(\tu{}) }^2.
\end{equation}
Similar to tensor train decomposition and matrix factorization,
the optimization \eqref{eq:trdobj} is a constrained non-convex
optimization problem.

Next, we would show that the tensor ring decomposition optimization
problem~\eqref{eq:trdobj} has spurious local minima even in a relaxed
constraint set. In the following, we will first construct an explicit
spurious local minimum for a specific $d$-th order tensor and then we
remark that the specific tensor can be generalized to a set of tensors
and the spurious local minimum exists for any tensor in the set.

To simplify the notations, we assume $\vec{r} = r$ and $\vec{n} =
n = r^2 + 1$. Here and in the rest of the paper, we abuse notations
$\vec{r} = r$ and $ \vec{n} = n$ meaning that $\vec{r} = (r, \dots,
r)$ and $\vec{n} = (n, \dots, n)$ respectively.  The choice of
$n=r^2+1$ comes from our specific construction of the target tensor:
If we consider a TR format with the bond dimension being $r$, $r^2$
is a large enough external dimension since the dimension of the space
spanned by $\tu{i}_{k_1,k_2},\ k_1,k_2=1,2,\dots,r$, is smaller than
or equal to $r^2$ for any $i=1,2,\dots,d$. Then for constructing the
target tensor, we add an additional orthogonal term, which enlarges
the bond dimension from $r$ to $r+1$ and the external dimension from
$r^2$ to $n=r^2+1$.  A discussion on $n \geq r^2+1$ cases is deferred
to the end of this section.

Further, we denote the \emph{lexicographical order} of a multi-tuple,
\textit{i.e.},
\begin{equation}
    \pi(\idxcont{k}{d}) \coloneq 1 + \sum_{i=1}^d (k_i - 1) r^{d-i}
\end{equation}
for $1 \leq \idxcont{k}{d} \leq r$. For example, for a $2$-tuple,
the lexicographical order function is $\pi(k_1,k_2) = (k_1 - 1) r +
k_2$ for $1 \leq k_1, k_2 \leq r$.

The specific $d$-th order tensor of bond dimension $r+1$ is
constructed as
\begin{equation} \label{eq:T0}
    \tT_0 = \sum_{\idxcont{k}{d} = 1}^r \left( \bigotimes_{i=1}^d
    e_{\pi(k_{i+1},k_i)} \right) + \bigotimes_{i=1}^d e_{n},
\end{equation}
where $e_j$ is an indicator vector of length $n$ with one at $j$-th
entry and zero at every other entry. By the definition of TR format,
we can see that $\tT_0 \in \calR_{r+1,n}^d$.

Let us consider a relaxed optimization problem~\eqref{eq:trdobj}
for $\tT_0$
\begin{equation} \label{eq:trdobj0} \min_{\tu{}\, \in\,
    \ocalU_{r^{d-1},n}^d} \frac{1}{2} \fnorm{ \tT_0 - \tau(\tu{}) }^2,
\end{equation}
where the constraint set $\ocalU_{r+1,n}^d$ is relaxed to
$\ocalU_{r^{d-1},n}^d$, \textit{i.e.}, the bond dimension is increased
from $r+1$ to $r^{d-1}$, which is much larger. For simplicity of
notation, we denote the objective function as $f_0(\tu{}) = \frac{1}{2}
\fnorm{\tT_0 - \tau(\tu{})}^2$.  Obviously, the objective function
at a global minimum is 0.

We now take a specific point $\tu{}_0 \in \ocalU_{r^{d-1},n}^d$ as
\begin{equation} \label{eq:u0}
    \begin{split}
        & \tu{}_0 = (\tu{1}, \tu{2}, \cdots, \tu{d}) \text{ with,} \\
        & \tu{i}_{\pi(\idxcont{p}{d-1}),\pi(\idxcont{q}{d-1})} =
        \delta_{p_1 q_1} \cdots \delta_{p_{i-1}q_{i-1}}
        \delta_{p_{i+1}q_{i+1}} \cdots \delta_{p_{d-1}q_{d-1}}
        e_{\pi(p_i,q_i)}, \text{ and,}\\
        & \tu{d}_{\pi(\idxcont{p}{d-1}),\pi(\idxcont{q}{d-1})} =
        \delta_{p_2 q_1} \cdots \delta_{p_{d-1} q_{d-2}}
        e_{\pi(p_1,q_{d-1})},
    \end{split}
\end{equation}
where $1 \leq \idxcont{p}{d-1}, \idxcont{q}{d-1} \leq r$ and $1 \leq
i \leq d-1$. $\tT_0$ and $\tu{}_0$ for $d=3$ and $r=2$ are illustrated
explicitly in Appendix~\ref{visualize}.

We first prove a property which is essential in the
discussions below.

\begin{proposition}\label{prop:u0}
    For $\tu{}_0$ defined in \eqref{eq:u0}, it holds that
    \begin{equation}
        \begin{split}
            \sum_{k_2, k_3, \dots, k_{d-1} = 1}^m
            & \tu{1}_{\pi(\idxcont{p}{d-1}),k_2} \otimes \tu{2}_{k_2, k_3}
            \otimes \\
            & \cdots \otimes \tu{d-2}_{k_{d-2}, k_{d-1}} \otimes
            \tu{d-1}_{k_{d-1}, \pi(\idxcont{q}{d-1})}
            =  \bigotimes_{i=1}^{d-1} e_{\pi(p_i,q_i)}
        \end{split}\label{eq:1to1-d}
    \end{equation}
    for any $1 \leq \idxcont{p}{d-1}, \idxcont{q}{d-1} \leq r$.
\end{proposition}

\begin{proof} [Proof of Proposition~\ref{prop:u0}] We can prove this
proposition by the direct computation:
\begin{equation*}
    \begin{split}
        &\sum_{k_2, k_3, \dots, k_{d-1}=1}^m
        \tu{1}_{\pi(p_1,\dots,p_{d-1}),k_2} \otimes
        \tu{2}_{k_2, k_3} \otimes \cdots \otimes
        \tu{d-1}_{k_{d-1},\pi(q_1,\dots,q_{d-1})}\\
        =&\sum_{k_3, \dots, k_{d-1}=1}^m\sum_{p_1'=1}^r e_{\pi(p_1,p_1')}
        \otimes \tu{2}_{\pi(p_1',p_2,\dots,p_d), k_d}\otimes \cdots
        \otimes \tu{d-1}_{k_{d-1},\pi(q_1,\dots,q_{d-1})}\\
        = & \cdots \\
        =&\sum_{p_1',\cdots,p_{d-2}'=1}^r e_{\pi(p_1,p_1')}
        \otimes  \cdots\otimes e_{\pi(p_{d-2},p_{d-2}')}\otimes
        \tu{d-1}_{\pi(p_1',\cdots,p_{d-2}',p_{d-1}),
        \pi(q_1,\cdots,q_{d-2},q_{d-1})}\\
        =&\sum_{p_1',\dots,p_{d-2}'=1}^r e_{\pi(p_1,p_1')} \otimes
        \cdots \otimes e_{\pi(p_{d-2},p_{d-2}')} \delta_{p_1' q_1}\cdots
        \delta_{p_{d-2}',q_{d-2}}e_{\pi(p_{d-1},q_{d-1})}\\
        =& e_{\pi(p_1,q_1)}\otimes e_{\pi(p_2,q_2)}\otimes\cdots\otimes
        e_{\pi(p_{d-1},q_{d-1})}.
    \end{split}
\end{equation*}
  
\end{proof}

Theorem~\ref{thm:locmin} below states that $\tu{}_0$ as in \eqref{eq:u0}
is a local minimum for \eqref{eq:trdobj0} with nonzero objective
function, \textit{i.e.}, the problem \eqref{eq:trdobj0} has a
spurious local minimum. Both $\tT_0$ in \eqref{eq:T0} and $\tu{}_0$
in \eqref{eq:u0} are abstract. The idea of $\tu{}_0$ is that, we want
to construct a TR format whose bond dimension is as large as possible
and that has properties similar to \eqref{eq:1to1-d} which is essential
in the proof of Theorem~\ref{thm:locmin}. Then we come to $\tu{}_0$ and
its corresponding bond dimension $r^{d-1}$. $\tT_0$ is then constructed
by adding an orthogonal term to $\tau(\tu{}_0)$.

We first provide the definition of local minimum used
in this paper.

\begin{definition} \label{def:locmin}
    $\tu{}$ is a local minimum of a real-valued function $f(\cdot)$
    if there exists a $\eta > 0$ such that for any $\norm{\tv{}} <
    \eta$, $f(\tu{}+\tv{}) \geq f(\tu{})$.
\end{definition}

Due to the non-strict inequality in Definition~\ref{def:locmin}, the
local minimum throughout this paper is also non-strict. $\norm{\cdot}$
can be any norm of tensor ring since norms are equivalent
in a finite-dimensional vector space. In the rest of this paper,
the maximum norm of tensor ring $\norm{\cdot}_{\max}$ is used for
simplicity, \textit{i.e.}, $\norm{\tv{}}_{\max}$ denotes the maximum
of the absolute values of entries in $\tv{}$, \textit{i.e.},
\begin{equation*}
    \norm{\tv{}}_{\max} \coloneq \max_{i,k,x,k'}\abs{\tv{i}(k,x,k')}.
\end{equation*}
A tensor ring $\tu{}$ is a \emph{spurious local minimum} of $f(\cdot)$
if $\tu{}$ is a local minimum and $f(\tu{}) > \min_{\tv{}} f(\tv{})$.

\begin{theorem} \label{thm:locmin}
    For $d\geq 3$, $\tu{}_0 \in \ocalU_{r^{d-1},n}^d$ as in
    \eqref{eq:u0} is a local minimum of problem \eqref{eq:trdobj0}
    and $f_0(\tu{}_0) > 0$. Hence $\tu{}_0$ is a spurious local minimum
    of \eqref{eq:trdobj0}.
\end{theorem}

The proof of Theorem~\ref{thm:locmin} consists of two parts. First
we demonstrate that $\tau(\tu{}_0)$ is the first summation part of
$\tT_0$ as in \eqref{eq:T0}.  Hence $f_0(\tu{}_0) = \frac{1}{2} >
0$. Next, we prove that there exists a small constant $\eta$ such that
for any $\norm{ \tv{} }_{\max} < \eta$, we have $f_0(\tu{}_0+\tv{})
\geq f_0(\tu{}_0)$. Therefore, $\tu{}_0$ is a local minimum in the
topology deduced by the norm $\norm{\tv{}}_{\max}$.  Since $\tT_0$
can be exactly represented by a tensor ring in $\ocalU_{r^{d-1},n}^d$,
$\tu{}_0$ is a spurious local minimum.

\begin{proof}[Proof of Theorem~\ref{thm:locmin}] 
For simplicity, we drop all subscripts $0$ of $\tT_0$, $\tu{}_0$ and
$f_0$. Further, we denote the bond dimension of $\ocalU_{r^{d-1},n}^d$
as $m = r^{d-1}$. Combining \eqref{eq:1to1-d} together with the
contraction of $\tu{d}$, we have
\begin{equation} \label{eq:diffTtau}
    \begin{split}
        &\qquad  \tau(\tu{}) = \sum_{\idxcont{k}{d} = 1}^m
        \tu{1}_{k_1, k_2} \otimes \tu{2}_{k_2, k_3} \otimes \cdots
        \otimes \tu{d}_{k_d k_1} \\ & =  \sum_{\idxcont{p}{d-1},
        \idxcont{q}{d-1} = 1}^r e_{\pi(p_1, q_1)} \otimes
        \cdots \otimes e_{\pi(p_{d-1}, q_{d-1})} \otimes
        \tu{d}_{\pi(\idxcont{q}{d-1}),\pi(\idxcont{p}{d-1})}\\ & =
        \sum_{\idxcont{q}{d-1}, p_{d-1} = 1}^r e_{\pi(q_2,q_1)}
        \otimes e_{\pi(q_3,q_2)} \otimes \cdots \otimes
        e_{\pi(q_{d-1}, q_{d-2})} \otimes e_{\pi(p_{d-1},q_{d-1})}
        \otimes e_{\pi(q_1,p_{d-1})}\\ & =  \tT - \bigotimes_{i=1}^d
        e_n.
    \end{split}
\end{equation}
Therefore, the objective function of $\tu{}$ is strictly positive,
\begin{equation}
    f(\tu{}) = \frac{1}{2} \fnorm{ \bigotimes_{i=1}^d e_n}^2 =
    \frac{1}{2} > 0.
\end{equation}
One more point about $\tau(\tu{})$, that will become important later,
is that $\tau(\tu{})$ has empty outer most layer, \textit{i.e.},
\begin{equation}\label{u-out-0}
    \tau(\tu{})(\idxcont{x}{d}) = 0, \quad \text{if } x_1 = n,
    \text{or } x_2 = n, \text{or } \dots, \text{or } x_d = n. 
\end{equation}
Hence $\tau(\tu{})$ is orthogonal to $\bigotimes_{i=1}^d e_n = \tT -
\tau(\tu{})$.

Next, we investigate the property of the neighborhood of $\tu{}$.
For any point $\tv{} \in \ocalU_{m,n}^d$, we have
\begin{equation} \label{eq:difff}
    \begin{split}
        f(\tu{}+\tv{}) - f(\tu{}) & = \frac{1}{2} \fnorm{ \tT -
        \tau(\tu{}+\tv{}) }^2 - \frac{1}{2} \fnorm{ \tT -
        \tau(\tu{}) }^2 \\
        & = \frac{1}{2} \fnorm{
        \tau(\tu{}+\tv{}) - \tau(\tu{}) }^2 - \left\langle \tT -
        \tau(\tu{}), \tau(\tu{}+\tv{}) -
        \tau(\tu{}) \right\rangle \\
        & = \frac{1}{2} \fnorm{ \tau(\tu{}+\tv{}) - \tau(\tu{})
        }^2 - \left\langle \bigotimes_{i=1}^d e_n, \tau(\tu{}+\tv{})
        \right\rangle\\
        & = \frac{1}{2} \fnorm{ \tau(\tu{}+\tv{}) - \tau(\tu{})
        }^2 -\sum_{\idxcont{k}{d} = 1}^m \prod_{i=1}^d
        \tv{i}(k_i,n,k_{i+1})\\
        & = \frac{1}{2} \fnorm{ \tau(\tu{}+\tv{}) - \tau(\tu{})
        }^2+O\left(\norm{\tv{}}_{\max,n}^3\right),
    \end{split}
\end{equation}
where the second equality is due to the definition of tensor Frobenius
norm as~\eqref{eq:fnorm}, the third equality adopts the result in
\eqref{eq:diffTtau} and orthogonality between $\tau(\tu{})$ and $\tT
- \tau(\tu{})$, the fourth equality comes from the direct evaluation
of the inner product, and $\norm{\tv{}}_{\max,n}$ is defined as
\begin{equation*}
    \norm{\tv{}}_{\max,n}=\max_{j,k,k'}\left\{\abs{
        \tv{j}(k, n, k')}\right\}\leq \norm{\tv{}}_{\max}.
\end{equation*}

In order to show that $\tu{}$ is a local minimum of $f$,
we need to show that the second-order term of $\frac{1}{2}
\fnorm{ \tau(\tu{}+\tv{}) - \tau(\tu{}) }^2$ has positive
coefficient. Denote $\tV$ as the difference between $d$-th order
tensor $\tau(\tu{}+\tv{})$ and $\tau(\tu{})$. For any fixed $p_i,
q_i$, $i=1,2,\dots,d-1$, the inner product of $\tV$ with $\left(
\bigotimes_{i=1}^{d-1}e_{\pi(p_i,q_i)} \right) \otimes e_n$ yields
\begin{equation}\label{eq:Velem}
    \begin{split}
        & \quad \biggl\langle \tV, \left(
        \bigotimes_{i=1}^{d-1}e_{\pi(p_i,q_i)} \right) \otimes e_n
        \biggr\rangle \\
        & = \left\langle \tau(\tu{}+\tv{}) - \tau(\tu{}), \left(
        \bigotimes_{i=1}^{d-1}e_{\pi(p_i,q_i)} \right) \otimes e_n
        \right\rangle \\
        & = \sum_{\idxcont{k}{d} = 1}^m \left\langle \bigotimes_{i=1}^d
        \left( \tu{i}_{k_i, k_{i+1}} + \tv{i}_{k_i, k_{i+1}} \right),
        \left( \bigotimes_{i=1}^{d-1}e_{\pi(p_i,q_i)} \right) \otimes
         e_n  \right\rangle \\
        & = \sum_{\idxcont{k}{d} = 1}^m \tv{d}(k_d, n, k_1) \cdot
        \prod_{i=1}^{d-1} \left( \tu{i}(k_i, \pi(p_i,q_i), k_{i+1})
        + \tv{i}(k_i, \pi(p_i,q_i), k_{i+1}) \right) \\
        & = \sum_{\idxcont{k}{d} = 1}^m \tv{d}(k_d,n,k_1) \cdot
        \prod_{i=1}^{d-1} \tu{i}(k_i, \pi(p_i,q_i), k_{i+1}) +
        O\left(\norm{\tv{}}_{\max,n}\cdot \norm{\tv{}}_{\max}\right)\\
        & =: S + O\left(\norm{\tv{}}_{\max,n}\cdot
        \norm{\tv{}}_{\max}\right),
    \end{split}
\end{equation}
where the third equality is due to \eqref{u-out-0} and the last
equality defines $S$. The definition \eqref{eq:u0} implies that for any
$i = 1, 2, \dots, d-1$, $\tu{i}(k_i, \pi(p_i, q_i), k_{i+1}) \neq 0$
if and only if all coordinates of $k_i$ and $k_{i+1}$ are the same ,
except the $i$-th one which is $p_i$ and $q_i$ for $k_i$ and $k_{i+1}$
respectively. Therefore, we know that
\begin{equation*}
    \prod_{i=1}^{d-1} \tu{i}(k_i, \pi(p_i, q_i), k_{i+1}) \neq 0,
\end{equation*}
if and only if $k_i = \pi(\idxcont{q}{i-1}, p_i, \dots, p_{d-1})$
for all $i = 1, \dots, d$. Hence the $S$ part can be rewritten as
\begin{equation} \label{eq:S2}
    S = \tv{d}(\pi(\idxcont{q}{d-1}), n,
    \pi(\idxcont{p}{d-1}))=O\left(\norm{\tv{}}_{\max,n}\right).
\end{equation}
Substituting \eqref{eq:S2} into \eqref{eq:Velem}, we obtain a lower
bound on the square of an element of $\tV$:
\begin{equation} \label{eq:sqrV}
    \begin{split}
         &\tV(\pi(p_1,q_1),\dots,\pi(p_{d-1},q_{d-1}),n)^2\\
         &= \left\langle \tV, \left(\bigotimes_{i=1}^{d-1}
         e_{\pi(p_i,q_i)} \right) \otimes e_n \right\rangle^2 =  S^2+
         O\left(\norm{\tv{}}_{\max,n}^2\cdot \norm{\tv{}}_{\max}\right)
         \\
         &= \abs{ \tv{d}(\pi(\idxcont{q}{d-1}), n,
         \pi(\idxcont{p}{d-1}) ) }^2 +
         O\left(\norm{\tv{}}_{\max,n}^2\cdot
         \norm{\tv{}}_{\max}\right).
    \end{split}
\end{equation}
Hence, we have 
\begin{equation} \label{eq:sumV}
    \begin{split}
        & \sum_{\idxcont{x}{d-1} = 1}^{n-1} \tV(\idxcont{x}{d-1},
        n )^2 \\ 
        & = \sum_{\idxcont{p}{d-1}, \idxcont{q}{d-1} =
        1}^r \tV(\pi(p_1,q_1), \dots, \pi(p_{d-1},q_{d-1}),n)^2\\ 
        & = \sum_{\idxcont{p}{d-1}, \idxcont{q}{d-1} =
        1}^r \abs{ \tv{d}(\pi(\idxcont{q}{d-1}), n,
        \pi(\idxcont{p}{d-1}) ) }^2\\
        &\qquad\qquad + O\left(\norm{\tv{}}_{\max,n}^2\cdot
        \norm{\tv{}}_{\max}\right)\\
        & = \sum_{k_1,k_d=1}^m \abs{ \tv{d}(k_d, n, k_1) }^2  +
        O\left(\norm{\tv{}}_{\max,n}^2\cdot \norm{\tv{}}_{\max}\right).
    \end{split}
\end{equation}
In the derivation from \eqref{eq:Velem} to \eqref{eq:sumV}, the only
step relies on index $d$ is \eqref{eq:S2} and all other steps can be
generalized to other index $j$ with the notion of periodic index. 
Notice that with periodic indexing, we have
\begin{equation} \label{eq:S2gen}
    \begin{split}
        & \prod_{i = j+1}^{j-1} \tu{i}(k_i,\pi(p_{i},q_{i}),k_{i+1}) \\
        & =  \prod_{i = j+1}^{j-1}
        \tu{i}(\pi(\idxcont{s^i}{d-1}),
        \pi(p_{i},q_{i}), \pi(\idxcont{s^{i+1}}{d-1})) \\
        & = \delta_{s^d_2, s^1_1} \cdots \delta_{s^d_{d-1}, s^1_{d-2}}
        \delta_{s^d_1, p_d}
        \delta_{s^1_{d-1}, q_d} \times  \\
        & \qquad \times \prod_{\substack{i=j+1 \\ i \neq d}}^{j-1} 
        \delta_{s^i_{1}, s^{i+1}_{1}} \cdots \delta_{s^i_{i-1},
        s^{i+1}_{i-1}}
        \delta_{s^i_{i+1}, s^{i+1}_{i+1}}
        \cdots \delta_{s^i_{d-1}, s^{i+1}_{d-1}} \delta_{s^i_i, p_i}
        \delta_{s^{i+1}_i, q_i} \\
    \end{split}
\end{equation}
where $\delta_{\cdot,\cdot}$ is the Kronecker delta, $k_i =
\pi(\idxcont{s^i}{d-1})$, and any $s$ is an integer in the interval
$[1,r]$ for $j\in\{1,\dots,d-1\}$. Through a careful derivation,
\eqref{eq:S2gen} is nonzero if and only if $k_{j} = \pi(q_1, \dots,
q_{j-1}, q_{j+1}, \dots, q_d)$, $k_{j+1} = \pi(p_d, p_1, \dots,
p_{j-1}, p_{j+1}, \dots, p_{d-1})$, and other $k_i$ is given in
Figure~\ref{fig:S2gen}.

\begin{figure}[ht]
    \centering
    \includegraphics[width=0.8\textwidth]{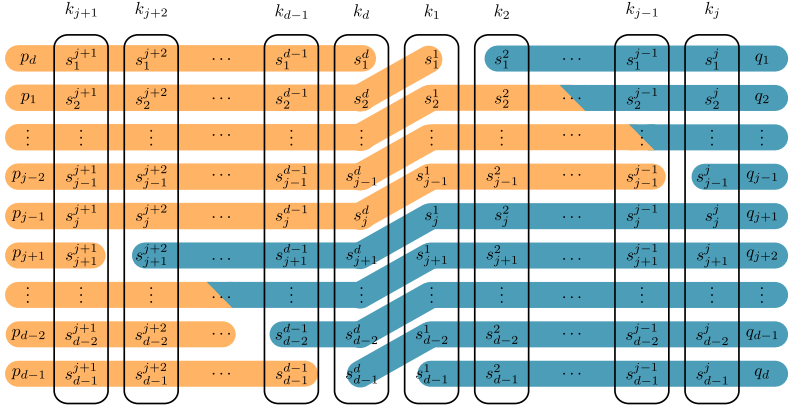}
    \caption{Derivation of generalized $S$ term. Orange lines and
    blue lines indicate equivalent relation of $p$ and $q$
    respectively.}
    \label{fig:S2gen}
\end{figure}

Similar to \eqref{eq:sqrV}, we can obtain a lower bound on the square
of an element of $\tV$
\begin{equation} \label{eq:sqrVgen}
    \begin{split}
        & \tV(\pi(p_1,q_1), \dots, \pi(p_{j-1},q_{j-1}), n,
        \pi(p_{j+1},q_{j+1}), \dots, \pi(p_d,q_d))^2\\ 
        & = \abs{ \tv{j}(\pi(q_1, \dots, q_{j-1},
        q_{j+1}, \dots, q_d), n,
        \pi(p_d, p_1, \dots, p_{j-1}, p_{j+1}, \dots, p_{d-1}) ) }^2
        \\
        &\qquad  + O\left(\norm{\tv{}}_{\max,n}^2 \cdot
        \norm{\tv{}}_{\max}\right).
    \end{split}
\end{equation}
Hence, similarly, we have 
\begin{equation}\label{eq:sumVgen}
    \begin{split}
        &\sum_{\idxcont{x}{j-1}, x_{j+1},\dots, x_d = 1}^{n-1}
        \tV(\idxcont{x}{j-1}, n, x_{j+1}, \dots, x_d)^2 \\
        &= \sum_{k_j,k_{j+1}=1}^m \abs{ \tv{j}(k_j, n, k_{j+1}) }^2 +
        O\left(\norm{\tv{}}_{\max,n}^2\cdot \norm{\tv{}}_{\max}\right),
    \end{split}
\end{equation} 
for any $j = 1, \dots, d$.

It follows from \eqref{eq:sumVgen} that
\begin{equation} \label{eq:boundV}
    \begin{split}
        \fnorm{\tV}^2 & \geq \sum_{j=1}^d \sum_{x_1, \dots, x_{j-1},
        x_{j+1}, \dots, x_d = 1}^{n-1} \tV(x_1, \dots, x_{j-1}, n,
        x_{j+1}, \dots, x_{d})^2 \\
        & = \sum_{j=1}^d \sum_{k_j, k_{j+1} = 1}^m \abs{ \tv{j}(k_j,
        n, k_{j+1})}^2 + O\left(\norm{\tv{}}_{\max,n}^2\cdot
        \norm{\tv{}}_{\max}\right)\\
        &\geq \left(1+O\left(\norm{\tv{}}_{\max}\right)\right)\cdot
        \norm{\tv{}}_{\max,n}^2.
    \end{split}
\end{equation}

Substituting \eqref{eq:boundV} into \eqref{eq:difff}, when $\eta$ is
sufficiently small, we have
\begin{equation}\label{eq:f(u+v)-f(u)>=0}
    \begin{split}
        f(\tu{}+\tv{}) - f(\tu{}) &= \frac{1}{2} \fnorm{
        \tau(\tu{}+\tv{}) - \tau(\tu{}) }^2 -\sum_{\idxcont{k}{d} =
        1}^m \prod_{i=1}^d \tv{i}(k_i,n,k_{i+1})\\
        &= \frac{1}{2} \fnorm{\tV}^2 +
        O\left(\norm{\tv{}}_{\max,n}^3\right) \\
        &\geq  \frac{1+O\left(\norm{\tv{}}_{\max}\right)}{2} \cdot
        \norm{\tv{}}_{\max,n}^2 + O\left(\norm{\tv{}}_{\max,n}^3\right)
        \geq 0,
    \end{split}
\end{equation}
for sufficiently small $\norm{\tv{}}_{\max}$.  Since the minimum value
of \eqref{eq:trdobj0} is zero, $\tu{}$ is a spurious local minimum.
\end{proof}

Theorem~\ref{thm:locmin} states that, for a particular $d$-th
order tensor as \eqref{eq:T0}, when the target tensor is of
TR format with bond dimension $r+1$, there exists at least one
spurious local minima even when the restricted bond dimension of
the problem \eqref{eq:trdobj} is $r^{d-1}$, much larger than $r+1$.
Through a more detailed analysis, we can locally illustrate that
when $f_0(\tu{}_0+\tv{}) = f_0(\tu{}_0)$ the constructed tensors are
identical, \textit{i.e.}, $\tau(\tu{}_0+\tv{}) = \tau(\tu{}_0)$.

\begin{proposition}\label{locmin:orbit}
    Under the same assumptions as in Theorem~\ref{thm:locmin}, for a
    positive constant $\eta$ which is small enough, any point $\tv{}
    \in \ocalU_{m,n}^d$ with $\norm{\tv{}}_{\max}<\eta$ satisfies
    $f_0(\tu{}_0+\tv{}) = f_0(\tu{}_0)$ if and only if $\tau(\tu{}_0 +
    \tv{}) = \tau(\tu{}_0)$, \textit{i.e.}, $\tu{}_0$ and $\tu{}_0+\tv{}$
    are tensor ring formats of the same $d$-th order tensor.
\end{proposition}

\begin{proof}[Proof of Proposition~\ref{locmin:orbit}]
It can be easily seen that $\tau(\tu{}_0 + \tv{}) = \tau(\tu{}_0)$
implies $f_0(\tu{}_0+\tv{}) = f_0(\tu{}_0)$.

Now we consider the reverse direction. Suppose that $\eta$ is small
enough. For any $\tv{}$ such that $\norm{\tv{}}_{\max}<\eta$
and $f_0(\tu{}_0+\tv{}) = f_0(\tu{}_0)$, the equality in
\eqref{eq:f(u+v)-f(u)>=0} implies that $\norm{\tv{}}_{\max,n}=0$,
\textit{i.e.},
\begin{equation}
    \tv{j}(k_j, n, k_{j+1}) = 0,\quad \forall\,1 \leq j \leq d, \,
    1 \leq k_j, k_{j+1} \leq m,
\end{equation}
which leads to
\begin{equation}
    f_0(\tu{}_0+\tv{}) - f_0(\tu{}_0) = \frac{1}{2} \fnorm{\tV}^2 =
    \frac{1}{2} \fnorm{ \tau(\tu{}_0+\tv{}) - \tau(\tu{}_0)}^2 = 0.
\end{equation}
Thus, we can conclude that $\tau(\tu{}_0 + \tv{}) = \tau(\tu{}_0)$.
\end{proof}

We remark that the same results as in Theorem~\ref{thm:locmin} and
Proposition~\ref{locmin:orbit} hold for $n \geq r^2+1$ if vectors
$e_k$ in the definition of $\tT_0$ and $\tu{}_0$ are extended to be of
length $n$ with zero padding.  Furthermore, Theorem~\ref{thm:locmin}
and Proposition~\ref{locmin:orbit} also hold for a generalized version
of $\tT_0$ and the associated $\tu{}_0$, where $\tT_0$ is defined as
\begin{equation}
    \tT_0 = \sum_{\idxcont{k}{d} = 1}^r \left(
    \bigotimes_{i=1}^d\lambda^i_{\pi(k_{i+1},k_i)}
    g^i_{\pi(k_{i+1},k_i)} \right) +
    \bigotimes_{i=1}^d\lambda^i_{r^2+1} g^i_{r^2+1},
\end{equation}
with $\lambda^i_k>0$ and $\langle g^i_{k_1},g^i_{k_2}\rangle =
\delta_{k_1 k_2}$ for any $1 \leq k_1, k_2 \leq r^2+1$ and $1
\leq i \leq d$, and $\tu{}_0$ is given via replacing $e_{k}$ in the
definition of $\tu{i}$ in \eqref{eq:u0} by $\lambda^i_{k} g^i_{k}$ for
$1\leq i\leq d$ and $1\leq k\leq r^2$. This is due to the orthogonal
rotation invariant property of Frobenius norm and an observation
that a scaling will not break the proof of Theorem~\ref{thm:locmin}
and Proposition~\ref{locmin:orbit} as long as $\eta>0$ is small enough.

Another remark is on the difficulty of TR decomposition.
Theorem~\ref{thm:locmin} and Proposition~\ref{locmin:orbit} show that
TR decomposition is more difficult than tensor train decomposition,
since TT decomposition has one-loop convergence~\cite{Holtz2012,
Rohwedder2013} if the restricted bond dimension equals to the
underlying bond dimension of the target TT format whereas local minima
remain in TR decomposition even if the restricted bond dimension
increases exponentially as in the size of the ring. When $d$ is large,
we may not expect a good landscape of TR decomposition even if it is
very much over-parameterized.

Finally, we also want to point out that,
if $\tu{}+\tv{}\in\mathcal{M}_{\tu{}}$, then
$\tau(\tu{}+\tv{})=\tau(\tu{})$, while it is not clear whether the
reverse holds in general. This difficulty comes from the ring structure
of TR format. In fact, for TT format, it can be proved that the gauge
invariant and the orbit with the same whole tensors are equivalent if
the TT format is full-rank~\cite{Rohwedder2013}.

\section{Alternating least square algorithm for tensor ring decomposition}
\label{sec: ALS}

In this section, we recall the \emph{alternating least square} (ALS)
algorithm for computing the tensor ring decomposition~\cite{Khoo2017a,
Wang2017} of a given $d$-th order tensor. Some basic descriptions
and properties are in Section~\ref{sec: intro ALS}. ALS is a strictly
monotonically descent algorithm unless a stationary point is found. In
Section~\ref{sec: one-loop}, we establish the convergence analysis of
ALS when the bond dimension is sufficiently higher than that of the
target tensor. As will be shown, in such cases, ALS converges in one
outer iteration, which is known as the \emph{one-loop convergence}.
Recall that in Theorem~\ref{thm:locmin} we prove the existence
of spurious local minima in the over-parameterized case. The bond
dimension required for the one-loop convergence is larger than that
in Theorem~\ref{thm:locmin} (much larger than that of the true tensor).

\subsection{Algorithm}
\label{sec: intro ALS}

Note that the objective function \eqref{eq:trdobj} is not convex due
to the multilinear mapping $\tau$. Nevertheless, if we fix all but
one of the $3$rd order tensors in $\tu{}$, \textit{e.g.}, $\tu{i}$,
and consider the suboptimization problem with respect to $\tu{i}$:
\begin{equation}\label{eq:obj_fun_coord}
    \min_{\tu{i}} \frac{1}{2} \fnorm{ \tT - \tau (\tu{1}, \dots,
    \tu{i-1}, \tu{i}, \tu{i+1}, \dots, \tu{d} ) }^2,
\end{equation}
this gives a quadratic least square problem in $\tu{i}$ and hence can
be solved explicitly and efficiently.

In order to make the least square formulation more explicit, we first
define a sequence of matrices $\{B_i\}_{i=1}^d$ of unfolded $\tT$
and two unfolding operators $\alpha(\cdot)$ and $\gamma(\cdot)$. In
the followings, the function $\pi$ always denotes the lexicographical
order. Given $1 \leq i \leq d$, $B_i$ is an unfolding of $\tT$ as a
matrix of size $( \prod_{j \neq i} n_j) \times n_i$, \textit{i.e.},
\begin{equation} \label{eq:Bi}
    B_i(\pi(x_{i+1}, \dots, x_{i-1}), x_i) =
    B_i(\pi(x_{i+1},\dots,x_d, x_1, \dots, x_{i-1}), x_i) :=
    \tT(x_1,\dots,x_d),
\end{equation}
where the periodic index convention is used and $1 \leq x_j \leq n_j$
for $j=1,\dots,d$.  The unfolding operator $\alpha(\cdot)$ unfolds
a sequence of 3rd order tensors into a matrix as,
\begin{equation} \label{eq:defalpha}
    A_i = \alpha(\tu{i+1},\dots,\tu{i-1})\in\bbR^{( \prod_{j \neq i}
    n_j) \times m^2},
\end{equation}
with entries being
\begin{multline}
    A_i(\pi(x_{i+1},\dots,x_{i-1}),\pi(k_{i+1},k_{i})) \\
    = \sum_{k_{i+2},\dots,k_{i-1} = 1}^m
    \tu{i+1}(k_{i+1},x_{i+1},k_{i+2}) \cdots
    \tu{i-1}(k_{i-1},x_{i-1},k_i),
\end{multline}
where $m$ is the bond dimension, $1 \leq x_j \leq n_j$ for
$j=1,\dots,d$, and $1 \leq k_i,k_{i+1} \leq m$.  The unfolding
operator $\gamma(\cdot)$ unfolds a 3rd order tensor into a matrix
with compatible indices of $A_i$ and $B_i$, \textit{i.e.},
\begin{equation} \label{eq:defgamma}
    X_i = \gamma(\tu{i})\in\bbR^{m^2\times n_i}
\end{equation}
with entries being,
\begin{equation}
    X_i(\pi(k_{i+1},k_{i}),x_i) = \tu{i}(k_{i},x_{i},k_{i+1}),
\end{equation}
where indices are in the same ranges as before. Since $\gamma(\cdot)$
acts on a single 3rd tensor, the unfolding operator is invertible
and the invert operator $\gamma^{-1}(\cdot)$ will be used in the
later content.

With these unfolded matrices, \eqref{eq:obj_fun_coord} then can be
rewritten as a standard least square problem
\begin{equation} \label{eq:LSPmat}
    \min_{X} \frac{1}{2} \fnorm{ A_i X - B_i }^2,
\end{equation}
where $A_i$ and $B_i$ are defined as \eqref{eq:defalpha}
and \eqref{eq:Bi} respectively, and the minimizer of
\eqref{eq:obj_fun_coord} can be achieved from $\gamma^{-1}(X_i)$
for $X_i$ being the minimizer of \eqref{eq:LSPmat}.

To simplify the notation, we denote the objective function in
\eqref{eq:obj_fun_coord} as $f_{\tT}$.  A popular numerical approach
for solving \eqref{eq:trdobj} is to solve \eqref{eq:obj_fun_coord}
for each of $\tu{1}$ to $\tu{d}$ in a cyclic way. The corresponding
algorithm is known as the alternating least square (ALS) algorithm for
tensor ring decomposition.  The pseudo-code of ALS for TR decomposition
is presented in Algorithm~\ref{alg:TR_ALS}, with the subscript $\ell$
indicating the iteration number.

\begin{algorithm}
    \caption{ALS for TR decomposition} \label{alg:TR_ALS}
    \begin{algorithmic}[1]
        \REQUIRE{Target $d$-th order tensor $\tT$ and initial tensor
        ring $\tu{}_0$.}
        \ENSURE{Converged tensor ring $\tu{}$.}
        \FOR{$\ell=0,1,2,\cdots$}
            \FOR{$i=1,2,\cdots,d$}
                \STATE Perform an ALS microstep:
                    \begin{equation*}
                        \tu{i}_{\ell+1} = \argmin_{\mathbf{v}}
                        \frac{1}{2} \fnorm{ \tT - \tau(\tu{1}_{\ell+1},
                        \dots, \tu{i-1}_{\ell+1}, \mathbf{v},
                        \tu{i+1}_{\ell}, \dots, \tu{d}_{\ell} ) }^2
                    \end{equation*}
            \ENDFOR
        \ENDFOR
    \end{algorithmic}
\end{algorithm}

The next lemma states that the objective function of $\tu{}_{\ell}$ in
Algorithm~\ref{alg:TR_ALS} decreases monotonically until a stationary
point is achieved.

\begin{lemma} \label{lem:mondecrease}
    Let $\{\tu{}_\ell\}_{\ell=0}^\infty$ be a sequence of
    tensor ring generated by Algorithm~\ref{alg:TR_ALS}.
    For any $\ell$, if $\tu{}_\ell$ is not a stationary point
    of $f_{\tT}$, \textit{i.e.}, $\grad f_{\tT}(\tu{}_\ell) \neq 0$, then
    $f_{\tT}(\tu{}_{\ell+1})<f_{\tT}(\tu{}_\ell)$.
\end{lemma}

\begin{proof}
    Since $\grad f_{\tT}(\tu{}_\ell) \neq 0$, there exists a set of
    indices $\calJ$ such that $\grad_{\tu{j}} f_{\tT}(\tu{}_\ell)
    \neq 0$ for $j \in \calJ$. Let $i$ be the smallest index
    in $\calJ$.  The $i$-th microstep solves a least square
    problem with nonzero gradient. Hence the objective value
    strictly decreases. For all later $(i+1)$-th to $d$-th
    microsteps the objective value is non-increasing due to the
    nature of least square solutions. Therefore, we conclude that
    $f_{\tT}(\tu{}_{\ell+1})<f_{\tT}(\tu{}_\ell)$ if $\tu{}_\ell$
    is not a stationary point of $f_{\tT}$.
\end{proof}

We would hope to get a stronger result than Lemma~\ref{lem:mondecrease}
showing that ALS converges to a stationary point. However, while ALS
has the monotonic descent, the convergence to a stationary point is
still open due to the following difficulties.

First, the boundedness of $\{\tu{}_\ell\}_{\ell=0}^\infty$ can not be
ensured. Thus, it is hard to say that $\{\tu{}_\ell\}_{\ell=0}^\infty$
has an accumulation point. To bypass the unboundedness, we can
instead consider the accumulation point of the sequence of manifolds
$\calM_{\tu{}_\ell}$. But it is also not clear that whether there
exists a sequence of gauge $\{\vec{A}_\ell\}_{\ell=0}^\infty$ such that
$\left\{\theta_{\vec{A}_\ell}(\tu{}_\ell)\right\}_{\ell=0}^\infty$
is bounded. If we consider the sequence of the whole tensor, it can
be proved that $\{\tau(\tu{}_\ell)\}_{\ell=0}^\infty$ is bounded. An
accumulation point of $\{\tau(\tu{}_\ell)\}_{\ell=0}^\infty$ may not
be located in $\calR_{\vec{r},\vec{n}}^d$ though since it is known
that the set of tensors in TR format with a fixed bond dimension is
not closed~\cite{Landsberg2012} due to the underlying ring structure.

Second, even if $\{\tu{}_\ell\}_{\ell=0}^\infty$ or
$\left\{\theta_{\vec{A}_\ell}(\tu{}_\ell)\right\}_{\ell=0}^\infty$
has an accumulation point $\tu{}$, it is hard to say that the rank
of $\tu{}$ is equal to the rank of $\tu{}_\ell$ when $\ell$ is large
enough. If equality does not hold, some continuity properties do
not hold at $\tu{}$, which leads to difficulties when analyzing
the limiting behavior of $\{\tu{}_\ell\}_{\ell=0}^\infty$ or
$\left\{\theta_{\vec{A}_\ell}(\tu{}_\ell)\right\}_{\ell=0}^\infty$.

In fact, if the boundedness and the equality of rank are assumed,
convergence to stationary point of ALS can be ensured, similar
to the analysis in~\cite{Espig2015}.  It is an interesting future
research direction to establish these conditions for the tensor ring
decomposition.

\subsection{One-loop convergence} \label{sec: one-loop}

Even though the general convergence result without the assumptions
above is still open, we can prove one-loop convergence of the
Algorithm~\ref{alg:TR_ALS} in an extremely over-parameterized
case, which means that Algorithm~\ref{alg:TR_ALS} converges in $d$
microsteps. In this section, we will prove the one-loop convergence of
Algorithm~\ref{alg:TR_ALS} under mild assumptions on the target tensor.

Let us consider the case that the target tensor admits a tensor ring
decomposition as
\begin{equation} \label{eq:trueT-oneloop}
    \tT =\tau(\tw{})= \sum_{\idxcont{k}{d}=1}^r \tw{1}_{k_1, k_2}
    \otimes \tw{2}_{k_2, k_3} \otimes \cdots \otimes \tw{d}_{k_d, k_1}
    \in \calR_{r,\vec{n}}^d,
\end{equation}
with bond dimension $r$ and tensor ring components $\tw{} =
\left(\tw{1}, \tw{2}, \dots, \tw{d}\right) \in \ocalU_{r,\vec{n}}^d$.
From Algorithm~\ref{alg:TR_ALS}, we notice that if both the target
tensor and the initial tensor ring are multiplied by an orthogonal
matrix on an external dimension, all iterators of the algorithm remain
the same up to the orthogonal transformation on the corresponding
external dimension. Hence we claim Algorithm~\ref{alg:TR_ALS} is
invariant under orthogonal transformations on external dimensions
and so is the related analysis.

We assume that the external dimension $\vec{n}$ is large enough
with $n_i\geq r^2$ for $i=1,\dots,d$ and we consider the problem
$\min\frac{1}{2} \fnorm{ \tT - \tau( \tu{} ) }^2$ where $\tu{} \in
\ocalU_{m,\vec{n}}^d$ with bond dimension $m=r^{d-1}$. Since the
dimension of $\spanfun{\bigl\{ \tw{i}_{k_1, k_2}: 1 \leq k_1, k_2
\leq r \bigr\}}$ is bounded above by $r^2$ and the invariant property
of Algorithm~\ref{alg:TR_ALS}, without loss of generality, we assume
that $\tw{}$ is located in a small subspace of $\ocalU^d_{r,\vec{n}}$:
\begin{equation}
    \ocalW^d_{r,\vec{n}} := \left\{ \tw{} \in \ocalU^d_{r,\vec{n}} \mid
    \tw{i}_{k_1,k_2}(s)=0,\ \forall\ 1\leq i\leq d,\ 1\leq k_1,k_2\leq
    r,\ s\geq r^2+1\right\},
\end{equation}
which is because that orthogonal transformations do not change the
Frobenius norm and implies that $\tT(x_1,\dots,x_d)=0$ as long as
one of $x_1,\dots,x_d$ is greater than or equal to $r^2+1$. Since
Frobenius norm is invariant under the orthogonal rotations in each
outer dimension of the target tensor, $\ocalW^d_{r,\vec{n}}$ can be
generalized to $\ocalU^d_{r,\vec{n}}$ under orthogonal rotations. Our
convergence result is the following theorem, in which $\mu$ denotes the
proper Lebesgue measure. The theorem guarantees one-loop convergence
for typical target tensor and initial guess, when the bond dimension
is large enough.

\begin{theorem} \label{thm:one-loop}
    There exists $\Omega_1 \subseteq \ocalW^d_{r,\vec{n}}$
    with $\mu(\Omega_1)=0$, such that for any $\tw{} \in
    \ocalW^d_{r,\vec{n}} \setminus \Omega_1$ and $\tT = \tau(\tw{})$,
    there exists $\Omega_2\subseteq \ocalU^d_{m,\vec{n}}$ with
    $\mu(\Omega_2)=0$, such that Algorithm~\ref{alg:TR_ALS} converges
    to the global minimum in $d$ microsteps as long as the initial
    point $\tu{}_0 \not \in \Omega_2$.
\end{theorem}

\begin{remark} \label{rmk:main_tech}
    In the proof of Theorem~\ref{thm:one-loop}, the technical
    part is to characterize two zero-measure sets $\Omega_1$
    and $\Omega_2$. Once these two sets are settled, the
    remaining proof is straightforward.  We postpone the proof of
    Theorem~\ref{thm:one-loop} towards the end of this section. In
    the following lemmas, we prove that a set is of zero measure
    through establishing the equivalence between this set and the set
    of roots of a polynomial, since the Lebesgue measure of the root
    set of a non-zero polynomial is zero.
\end{remark}

We first focus on the characterization $\Omega_2$ which
will be defined in Lemma~\ref{lem:exist_F}; $\Omega_1$ will
be characterized along the analysis and will be defined in
Lemma~\ref{lem:Omega1}. Both $\Omega_1$ and $\Omega_2$ are constructed
somewhat implicitly. Intuitively, $\Omega_2$ is the set of $\tu{}_0$
which leads to some degeneracy in the first $d$ microsteps of ALS. We
show that $\Omega_2$ is zero-measure following the argument sketched
in Remark~\ref{rmk:main_tech}, which would require some assumptions on
$\tT$. Then we denote $\Omega_1$ as the set of $\tw{}$ such that at
least one of those assumptions is violated and prove that $\Omega_1$
is also of zero-measure.

For simplicity, in the rest of this section we denote $\tu{}=\tu{}_0$
and $\tv{}=\tu{}_1$ as the initial tensor vector and the tensor
vector after one macro step (one loop).  Each micro step of
Algorithm~\ref{alg:TR_ALS} solves a least square problem as
\eqref{eq:LSPmat}. The full-column-rankness of $A_i$ leads to
the uniqueness of the solution, which is crucial for the one-loop
convergence. Thus, we lay down these natural assumptions for $j =
1,2,\dots,d$:
\begin{equation} \label{eq:A.j} \tag{A.j}
    A_j = \alpha(\tu{j+1},\dots,\tu{d},\tv{1},\dots,\tv{j-1}) \text{
    has full column rank}.
\end{equation}
Once \eqref{eq:A.j} is satisfied, \eqref{eq:LSPmat} has a unique
minimizer and $\tv{j}$ can be uniquely determined. 

All later proofs rely on a homogeneity property as defined in
Definition~\ref{def:multi-homo}.

\begin{definition} \label{def:multi-homo}
    A multi-variable function mapping from $d$ Euclidean spaces to
    an Euclidean space, \textit{i.e.},
    \begin{equation*}
        \begin{split}
            f : \ \bbR^{p_1} \times \bbR^{p_2} \times \dots
            \times \bbR^{p_d} \rightarrow & \quad\quad \bbR^{q}\\
            (x_1,x_2,\dots,x_d)\quad\ \ \mapsto& 
            f(x_1,x_2,\dots,x_d),\\
        \end{split}
    \end{equation*}        
    is \emph{\mhp{}} if 
    for any $\lambda\in\bbR$ and index $j\in\{1,2,\dots,d\}$,
    \begin{equation} \label{eq:multi-homo}
        f(x_1,\dots,x_{j-1},\lambda x_j, x_{j+1}, \dots, x_d) =
        \lambda^{s_j}
        f(x_1,\dots, x_d),
    \end{equation}
    where $s_j$ is a $j$-dependent non-negative integer, and each
    entry of $f(x_1, \dots, x_d)$ is a polynomial of entries of
    $x_1,x_2,\dots,x_d$.
\end{definition}

We emphasize that all $p_1, \dots, p_d$ and $q$ in
Definition~\ref{def:multi-homo} are multi-index notations,
\textit{e.g.}, $q = 2\times 3$. This means both the input $x_i$ and
the output $f(x_1, \dots, x_d)$ could be scalars, matrices, or tensors.

Next, we list two properties of \mhp{} function without detailed
proof. Concrete examples for both properties are given in
Appendix~\ref{app:mhp}.
\begin{itemize}
    \item (Productivity) The product of two \mhp{} functions is a \mhp{}.
        This product includes entry-wise product as well as compatible
        tensor contractions.
    \item (Composition) The composition of a \mhp{} function with a
        \mhp{} function is \mhp{}.
\end{itemize}

Obviously, unfolding operator $\alpha(\cdot)$ is a \mhp{} function.
In Lemma~\ref{lem:soln_multi_homo}, we show that under condition
\eqref{eq:A.j} and with proper scaling, the function mapping from
$(\tu{j+1},\dots,\tu{d}, \tv{1},\dots,\tv{j-1})$ to $\tv{j}$ can
be described by a \mhp{} function.  In Lemma~\ref{lem:exist_F}, we
characterize $\Omega_2$ and show the existence of a \mhp{} function
whose root set equals $\Omega_2$.

\begin{lemma} \label{lem:soln_multi_homo}
    There exists a \mhp{} function 
    \begin{equation*}
        G_j:\bigtimes_{i=j+1}^{j-1} \bbR^{m\times n_i \times
        m}\rightarrow \bbR^{m\times n_j \times m},
    \end{equation*}
    such that for any $(\tu{j+1},\dots,\tu{d}, \tv{1},\dots,\tv{j-1})$
    with \[A_j = \alpha(\tu{j+1},\dots,\tu{d}, \tv{1},\dots,\tv{j-1})\]
    satisfying condition \eqref{eq:A.j}, it holds that
    \begin{equation}
        \tv{j} = \frac{1}{\det{A_j^\top A_j}} G_j( \tu{j+1},\dots,\tu{d},
        \tv{1},\dots,\tv{j-1})
    \end{equation}
    for $\tv{j}$ being the solution of \eqref{eq:obj_fun_coord} at $j$-th
    microstep.
\end{lemma}

\begin{proof} [Proof of Lemma~\ref{lem:soln_multi_homo}]
    As $\alpha(\cdot)$ and $\alpha^\top(\cdot)$ are \mhp{} functions,
    $A_j^\top A_j$ is a \mhp{} function, due to the productivity
    property of \mhp{} functions. According to the definition
    of adjugate operation, we have $\mathrm{adj}(\lambda A) =
    \lambda^{m^2-1} \mathrm{adj}(A),\ \forall\ \lambda\in\bbR,\
    A\in\bbR^{m^2\times m^2}$ and each entry of $\mathrm{adj}(A)$
    is a polynomial of entries of $A$. Hence, due to composition
    property, $\mathrm{adj}(A_j^\top A_j)$ is a \mhp{} function. The
    folding operator $\gamma^{-1}$ is also \mhp{}, \textit{i.e.},
    $\gamma^{-1}(\lambda X) = \lambda \gamma^{-1}(X)$ for any $\lambda$
    and $X$. Applying the productivity property and composition
    property of \mhp{} function again, we have that
    \begin{equation}
        G_j( \tu{j+1},\dots,\tu{d},\tv{1},\dots,\tv{j-1})
        = \gamma^{-1} (\mathrm{adj}(A_j^\top A_j) A_j^\top B_j)
    \end{equation}
    is a \mhp{} function.
    
    When $A_j$ has full column rank, \textit{i.e.}, $\det{A_j^\top
    A_j} \neq 0$, the unique minimizer of \eqref{eq:LSPmat} can be
    written as,
    \begin{equation}
        X_j = (A_j^\top A_j)^{-1} A_j^\top B_j = \frac{1}{\det{A_j^\top
        A_j}} \mathrm{adj}(A_j^\top A_j) A_j^\top B_j,
    \end{equation}
    where $\mathrm{adj}(\cdot)$ denotes the adjugate. Hence,
    \begin{equation}
        \tv{j} = \gamma^{-1}(X_j) = \frac{1}{\det{A_j^\top A_j}} G_j(
        \tu{j+1},\dots,\tu{d}, \tv{1},\dots,\tv{j-1})
    \end{equation}
    proves the lemma.
\end{proof}

\begin{lemma} \label{lem:exist_F}
    Denote $\Omega_2 = \bigl\{(\tu{1},\dots,\tu{d}) \mid \text{at least
    one of \eqref{eq:A.j} fails for } j=1,2,\dots,d \bigr\}$. There
    exists a \mhp{} function $F:\ocalU^d_{m,\vec{n}} \rightarrow \bbR$,
    such that $\Omega_2$ is the root set of $F$.
\end{lemma}

\begin{proof}[Proof of Lemma~\ref{lem:exist_F}]
    Lemma~\ref{lem:exist_F} is proven by induction from $d$ down
    to $1$.  We first define a sequence of set $\Omega^{[j]}$ for $j =
    1,2,\dots,d$ via
    \begin{equation*}
        \Omega^{[j]} = \{(\tu{j},\dots,\tu{d},\tv{1},\dots,\tv{j-1})
        \mid \text{at least one of \hyperref[eq:A.j]{(A.i)} fails for }
        i=j,\dots,d \}.
    \end{equation*}
    Notice that $\Omega_2 = \Omega^{[1]}$.

    First consider $\Omega^{[d]}$. If \hyperref[eq:A.j]{(A.d)} fails,
    \textit{i.e.}, $A_d = \alpha (\tv{1},\dots,\tv{d-1})$ does not have
    full column rank, then $\det{A_d^\top A_d} = 0$ defines a \mhp{}
    function
    \begin{equation}
        F_d(\tu{d},\tv{1},\dots,\tv{d-1}) = \det{A_d^\top A_d},
    \end{equation}
    with $\tu{d}$ being dummy variable, such that the root set of $F_d$
    equals $\Omega^{[d]}$.

    Now, we take induction step. Assume
    for $j+1$, there exists a \mhp{} function
    $F_{j+1}(\tu{j+1},\dots,\tu{d},\tv{1},\dots,\tv{j})$ such that the
    root set of $F_{j+1}$ equals $\Omega^{[j+1]}$. The construction
    of $F_j$ can be divided into two scenarios: \eqref{eq:A.j} fails
    and \eqref{eq:A.j} holds.

    When \eqref{eq:A.j} fails, \textit{i.e.}, $A_j = \alpha
    (\tu{j+1},\dots,\tu{d},\tv{1},\dots,\tv{j-1})$ does not have
    full column rank, then $\det{A_j^\top A_j} = 0$ defines a \mhp{}
    function
    \begin{equation}
        H_j^1(\tu{j},\dots,\tu{d},\tv{1},\dots,\tv{j-1})
        = \det{A_j^\top A_j},
    \end{equation}
    with $\tu{j}$ being dummy variable, such that the root set of
    $H_j^1$ equals $\Omega^{[j]}$ given \eqref{eq:A.j} fails.

    When \eqref{eq:A.j} holds, according to
    Lemma~\ref{lem:soln_multi_homo}, there exists a \mhp{} function
    $G_j$ such that
    \begin{equation} \label{eq:tvj}
        \tv{j} = \frac{1}{\det{A_j^\top A_j}} G_j ( \tu{j+1},\dots,\tu{d},
        \tv{1},\dots,\tv{j-1}).
    \end{equation}
    Since \eqref{eq:A.j} holds, at least one of
    \hyperref[eq:A.j]{(A.i)} for $i=j+1,\dots,d$ fails and, hence,
    $F_{j+1}$ exists. Substituting \eqref{eq:tvj} into $F_{j+1}$ and
    multiplied by $\det{A_j^\top A_j}^s$, where $s$ is the homogeneity
    degree of the last variable, gives a \mhp{} function
    \begin{multline}
        H_j^2(\tu{j},\dots,\tu{d},\tv{1},\dots,\tv{j-1}) \\
        =  F_{j+1}(\tu{j+1},\dots,\tu{d},\tv{1},\dots,\tv{j-1},G_j(
        \tu{j+1},\dots,\tu{d}, \tv{1},\dots,\tv{j-1})).
    \end{multline}
    The root set of $H_j^2$ equals $\Omega^{[j]}$ given \eqref{eq:A.j}
    holds.
    
    Combining two scenarios together, we define the \mhp{} function
    \begin{equation}
        F_j = H_j^1 \cdot H_j^2,
    \end{equation}
    with root set equals $\Omega^{[j]}$.
    
    Finally, setting $F=F_1$ completes the proof.
\end{proof}
  
Next we prove that the \mhp{} function in Lemma~\ref{lem:exist_F}
is not constantly zero, which is the second step of the strategy
described in Remark~\ref{rmk:main_tech}. We need some mild assumptions
on the target tensor $\tT$. Let $\tT^j \in \bbR^{r^{d-2} \times r^2
\times r^d}$ be a reshape of nonzeros of $\tT$ for $j=1,2,\cdots,d-1$
satisfying
\begin{multline}
    \tT^j(\pi(p_1, \dots, p_{j-1}, q_{j+1}, \dots, q_{d-1}),
    :, \pi(q_1, \dots, q_j, p_j, \dots, p_{d-1}))\\
    =  \tT( \pi(p_1, q_1), \dots, \pi(p_{d-1}, q_{d-1}), 1:r^2),
\end{multline}
where $1 \leq p_i, q_i \leq r$ and $1 \leq i \leq d-1$. The mild
assumptions state as
\begin{equation} \label{eq:B.j} \tag{B.j}
\begin{split}
    &\tT^j(:,:,\pi(q_1,\cdots,q_j,p_j,\cdots,p_{d-1})),\ 1 \leq q_1,
    \dots, q_j, p_j, \dots, p_{d-1} \leq r, \\
    &\text{ are linearly independent},
\end{split}
\end{equation}
for $j=1,2,\dots,d-1$. $\Omega_1$ is the set of tensors violating
these assumption. Later we will prove that $\Omega_1$ has zero measure
in Lemma~\ref{lem:Omega1}.

Given these assumptions, we can show that the \mhp{} function in
Lemma~\ref{lem:exist_F} is not constantly zero.

\begin{lemma} \label{lem:Omega2}
    Suppose $\tw{}\in\ocalW^d_{r,\vec{n}}$ and \eqref{eq:B.j} holds for
    $j=1,2,\dots,d-1$. The \mhp{} function $F$ in
    Lemma~\ref{lem:exist_F} is not constant zero and $\Omega_2$ has zero
    Lebesgue measure.
\end{lemma}

\begin{proof} [Proof of Lemma~\ref{lem:Omega2}]
    First of all, $F$ is a polynomial of entries of
    $\tu{1},\dots,\tu{d}$. Then showing a polynomial is not constantly
    zero, it is sufficient to show that there exists a $\tu{}$ such that
    $F(\tu{}) \neq 0$ which is equivalent to show that for this $\tu{}$
    the condition \eqref{eq:A.j} holds for any $j=1,2,\cdots,d$.
    
    Let each 3-rd order tensor of $\tu{} = (\tu{1}, \tu{2}, \cdots,
    \tu{d}) \in \ocalU^d_{m,\vec{n}}$ be
    \begin{equation} \label{eq:u}
        \tu{i}_{\pi(\idxcont{p}{d-1}),\pi(\idxcont{q}{d-1})} =
        \delta_{p_1 q_1} \cdots \delta_{p_{i-1}q_{i-1}}
        \delta_{p_{i+1}q_{i+1}} \cdots \delta_{p_{d-1}q_{d-1}}
        e_{\pi(p_i,q_i)},
    \end{equation}
    for $1 \leq i \leq d-1$, and 
    \begin{equation}
        \tu{d}_{\pi(\idxcont{p}{d-1}),\pi(\idxcont{q}{d-1})} =
        \tT(\pi(q_1,p_1),\pi(q_2,p_2),\cdots,\pi(q_{d-1},p_{d-1}),:),
    \end{equation}
    where $1 \leq p_1,\dots,p_{d-1}, q_1,\dots,q_{d-1} \leq r$.
  
    Since $\tu{i}$ is defined identical to that in \eqref{eq:u0} for
    $1 \leq i \leq d-1$, then \eqref{eq:1to1-d} holds here as well.
    Hence, it is easy to verify that $\tau(\tu{}) = \tT$. Since $\tu{}$
    is already  a minimizer, we have $\tv{}=\tu{}$ if \eqref{eq:A.j}
    holds for all $j$, \textit{i.e.},  the minimizer in each microstep is
    unique. Thus, to prove that \eqref{eq:A.j} holds is equivalent to
    say that
    \begin{equation} \label{eq:tensorprodj}
        \sum_{k_{j+2},\dots,k_{j-1}=1}^m \tu{j+1}_{k_{j+1}, k_{j+2}}
        \otimes \cdots \otimes \tu{j-1}_{k_{j-1}, k_j},
    \end{equation}  
    $k_{j+1},k_j=1,2,\dots,m,$ are linearly independent, since
    each column of $A_j$ is an unfolding of \eqref{eq:tensorprodj}.
    Due to \eqref{eq:1to1-d}, \hyperref[eq:A.j]{(A.d)} holds directly.

    Consider a fixed $j\in\{1,2,\cdots,d-1\}$, and denote
    $k_j=\pi(p_1,p_2,\cdots,p_{d-1})$ and
    $k_{j+1}=\pi(q_1,q_2,\cdots,q_{d-1})$. With a careful index check,
    \eqref{eq:tensorprodj} equals
    \begin{equation}
        \begin{split}
            & \sum_{q'_{j+1},\dots,q'_{d-1},p'_1,\dots,p'_{j-1}=1}^r
            e_{\pi(q_{j+1},q'_{j+1})} \otimes \cdots \otimes
            e_{\pi(q_{d-1},q'_{d-1})} \\
            & \otimes \tT(\pi(p'_1,q_1),\dots,\pi(p'_{j-1},q_{j-1}),
            \pi(p_j,q_j),\pi(p_{j+1},q'_{j+1}),\dots,
            \pi(p_{d-1},q'_{d-1}),:) \\
            & \otimes e_{\pi(p'_1,p_1)} \otimes \cdots \otimes
            e_{\pi(p'_{j-1},p_{j-1})}.
        \end{split}
    \end{equation}
    Hence \eqref{eq:A.j} holds for $\tu{}$ is equivalent to
    \eqref{eq:B.j}. We have showed that $F(\tu{}) \neq 0$.

    Since $\Omega_2$ is the root set of $F$, which is a non-zero
    polynomial, the measure of $\Omega_2$ is zero.
  \end{proof}

If we merge the first and the second index of $\tT^j$ together,
then $\tT^j$ becomes an $r^d \times r^d$ matrix, denoted by
$\widetilde{\tT}^j$. Assumption \eqref{eq:B.j} is equivalent to say
that the matrix is full-rank. In random matrix theory, we know that the
measure of degenerate matrices is zero. Here, Lemma~\ref{lem:Omega1}
points out that, similar conclusion holds for $\tT^j$, if $\tw{1},
\dots, \tw{d}$ are generated randomly.

\begin{lemma} \label{lem:Omega1}
    Let $\Omega_1$ be the set of failure of \eqref{eq:B.j}, \textit{i.e.},
    \begin{equation*}
	    \Omega_1 = \{ \tw{} \in \ocalW^d_{r,\vec{n}} \mid \text{at least
	    one of \eqref{eq:B.j} is not satisfied for } j = 1, 2,
	    \dots, d-1 \},
    \end{equation*}
    Then $\Omega_1$ has Lebesgue measure 0.
\end{lemma}

\begin{proof} [Proof of Lemma~\ref{lem:Omega1}]
    Since $d$ is a finite number, it is sufficient to prove that for
    any $j = 1, 2, \dots, d-1$,
    \begin{equation} \label{eq:muCj}
        \mu(\{ \tw{} \mid \text{\eqref{eq:B.j} is not satisfied} \}) = 0.
    \end{equation}
    Consider a fixed $j \in \{1, 2, \dots, d-1 \}$. For any matrix
    $X \in \bbR^{r^d \times r^d}$, $X$ is column-rank-deficient,
    if and only if $\det{X^T X} = 0$. Similar to the proof
    of Lemma~\ref{lem:exist_F}, it can be shown that $f_j =
    \det{\left(\widetilde{\tT}^j\right)^T \widetilde{\tT}^j}$ is
    a \mhp{} polynomial. The condition in Lemma~\ref{lem:Omega1}
    can be rewritten in terms of \mhp{} polynomial, \textit{i.e.},
    the set in \eqref{eq:muCj} can be restated as,
    \begin{equation}
        \{ \tw{} \mid \text{\eqref{eq:B.j} is not satisfied} \} = \{ \tw{}
        \mid f_j(\tw{}) = 0 \}.
    \end{equation}
    If $f_j$ is not a zero
    polynomial, then the set of its roots has Lebesgue measure zero.
    
    The rest of the proof states that $f_j$ is not a zero polynomial.
    Consider the point
    \begin{equation}
        \tw{i}_{k_1, k_2} = e_{\pi(k_2,k_1)}, \quad 1 \leq k_1,
        k_2 \leq r, 1\leq i \leq d,
    \end{equation}
    which results $\tT = \sum_{\idxcont{k}{d} = 1}^r \bigotimes_{i=1}^d
    e_{\pi(k_{i+1}, k_i)}$ and the reshaped tensor,
    \begin{multline}
        \tT^j(\pi(p_1,\dots,p_{j-1},q_{j+1},\dots,q_{d-1}),
        :,\pi(q_1,\dots,q_j,p_j,\dots,p_{d-1}))\\
        = \delta_{p_1 q_2} \cdots \delta_{p_{j-1}q_j}
        \delta_{q_{j+1}p_j} \cdots \delta_{q_{d-1}p_{d-2}}
        e_{\pi(q_1,p_{d-1})}\in\bbR^{r^2},
    \end{multline}
    whose corresponding $r^d\times r^d$ matrix $\widetilde{\tT}^j$
    is an identity matrix. Since \eqref{eq:B.j} is satisfied for this
    specific $\tw{}$, $f_j$ is not a zero polynomial.
\end{proof}

Finally, with above technical lemmas, we now prove
Theorem~\ref{thm:one-loop}.

\begin{proof} [Proof of Theorem~\ref{thm:one-loop}]
    Let $\Omega_1$ and $\Omega_2$ be the measure-zero sets as defined in
    Lemma~\ref{lem:Omega1} and Lemma~\ref{lem:exist_F} respectively. We
    consider points $\tw{}\notin\Omega_1$ and $\tu{}\notin\Omega_2$. Set
    \begin{equation*}
      \bbX_i:=\text{span}\left\{ \tw{i}_{k_1, k_2}:
        k_1,k_2=1,2,\cdots,r\right\}
    \end{equation*}
    and denote $\bbP_i$ as the orthogonal projector onto $\bbX_i$ for
    $i=1,2,\cdots,d$.
    
    According to \eqref{eq:LSPmat}, we have
    \begin{multline*}
        \frac{1}{2} \fnorm{ \tT - \tau( \tx{1}, \tu{2}, \dots,
        \tu{d} )}^2 
        \geq \frac{1}{2} \fnorm{ \tT - \tau(
        \gamma^{-1}(\bbP_1(\gamma(\tx{1}))), \tu{2}, \dots, \tu{d})}^2,
    \end{multline*}
    for any $\tx{1}$, which implies that $\tv{1}_{k_1, k_2} \in \bbX_1,\
    \forall \ k_1,k_2=1,2,\dots,m$. Similarly, we have
    \begin{equation*}
        \tv{i}_{k_1, k_2} \in \bbX_i, \quad \forall\ k_1,k_2=1,2,\dots,m,
    \end{equation*}
    for $i = 1, 2, \dots, d-1$.

    Now we consider the last microstep
    \begin{equation}
        \tv{d} = \argmin_{\tx{d}} \frac{1}{2} \fnorm{ \tT - \tau(
        \tv{1}, \dots, \tv{d-1}, \tx{d} )}^2.
    \end{equation}
    The condition \hyperref[eq:A.j]{(A.d)} implies that 
    \begin{equation*}
        \sum_{k_2,\cdots,k_{d-1}=1}^m \tv{1}_{k_1, k_2} \otimes \cdots
        \otimes \tv{d-1}_{k_{d-1}, k_d},\quad  k_1,k_d=1,\dots,m,
    \end{equation*}
    are linearly independent. Combining the linear independence with
    \begin{equation}
        m^2 = r^{2(d-1)} \geq \prod_{k=1}^{d-1}\mathrm{dim}(\bbX_k) =
        \mathrm{dim}(\bbX_1 \otimes \cdots \otimes \bbX_{d-1}),
    \end{equation}
    and
    \begin{equation}
        \sum_{k_2,\cdots,k_{d-1}=1}^m \tv{1}_{k_1, k_2} \otimes \cdots
        \otimes \tv{d-1}_{k_{d-1}, k_d} \in \bbX_1 \otimes \cdots
        \otimes \bbX_{d-1}, \quad \forall\ k_1,k_d=1,\dots,m,
    \end{equation}
    we obtain that
    \begin{equation}
        \mathrm{span}\left\{ \sum_{k_2,\dots,k_{d-1}=1}^m \tv{1}_{k_1,
        k_2}\otimes \cdots \otimes \tv{d-1}_{k_{d-1}, k_d},\ 
        k_1,k_d=1,\dots,m \right\} = \bbX_1 \otimes \cdots \otimes
        \bbX_{d-1}.
    \end{equation}
    Noticing that $\tT \in \bbX_1 \otimes \cdots \otimes \bbX_{d-1}
    \otimes \bbR^{n_d}$, there exists $\tv{d}$ such that
    \begin{equation}
        \tau(\tv{1},\dots,\tv{d-1},\tv{d}) = \tT,
    \end{equation}
    which proves the theorem.
\end{proof}

\section{Numerical results}
\label{sec: numerical}

In this section, we present two sets of numerical results
to validate and further support Theorem~\ref{thm:locmin} and
Theorem~\ref{thm:one-loop} respectively.  For Theorem~\ref{thm:locmin},
in Section~\ref{sec: spurious local minima}, we identify
a non-strict spurious local minimum for a target tensor of
bond dimension $r+1$ as in \eqref{eq:T0} given the optimization
problem in $\ocalU_{r^{d-1},n}^d$. In this section, we numerically
validate that ALS algorithm in some sense can not escape from
the spurious local minimum in Theorem~\ref{thm:locmin} though
Proposition~\ref{locmin:orbit} suggests that this local minimum
might not be strict.  Whereas for Theorem~\ref{thm:one-loop},
in Section~\ref{sec: ALS}, the one-loop convergence is proven for
target tensors with bond dimension $r$ given the optimization problem
in $\ocalU_{r^{d-1},n}^d$ solved via ALS. We numerically test the
tightness of the bond dimension $r^{d-1}$ and show that the one-loop
convergence does not hold when the tensor ring space is reduced to
$\ocalU_{r^{d-1}-1,n}^d$.  All numerical results in this section are
generated from codes implemented and executed with MATLAB.

\subsection{The stability of the spurious local minimum}
\label{sec: stable local min}

Theorem~\ref{thm:locmin} shows that for the given target tensor
as \eqref{eq:T0}, there exists a carefully designed local minimum.
Due to the intrinsic difficulty of tensor ring format, the theorem does
not characterize the neighborhood of all tensor rings with equivalent
format.  Hence we numerically demonstrate that the designed local
minimum is somewhat a numerically inescapable local minimum for ALS.

We construct the target tensor as \eqref{eq:T0} with $d = 3$,
$r = 3$, and $n = r^2+1 = 10$ and the local minimum $\tu{}_0$
as \eqref{eq:u0} is constructed accordingly. For the purpose of
this section, we apply ALS with a initial tensor ring being a
perturbation of $\tu{}_0$.  The perturbation is added as follows.
Given a perturbation size $c \geq 0$, we add independent random
numbers, uniformly distributed on $[-c,c]$, on each entry of
$\tu{}_0$. We select 1000 choices of $c$ between $0$ and $0.3$. For
each $c$, $10^5$ perturbations are tested via ALS and the least-square
problems therein are solved via MATLAB backslash. When the objective
functions of a converged iteration stay above $\frac{1}{2}$, we
claim it is trapped by the local minimum. Otherwise, it escapes
from the local minimum. Figure~\ref{fig:trap_escape} shows two
typical convergence behaviors of a ``trapping'' and an ``escaping''
iteration. Figure~\ref{fig:locminphase} plots the phase transition
of the empirical probability that ALS is trapped at a TR format which
leads to the same whole tensor as the local minimum $\tu{}_0$.

\begin{figure}[ht]
	\centering
    \includegraphics[height=0.5\textwidth]{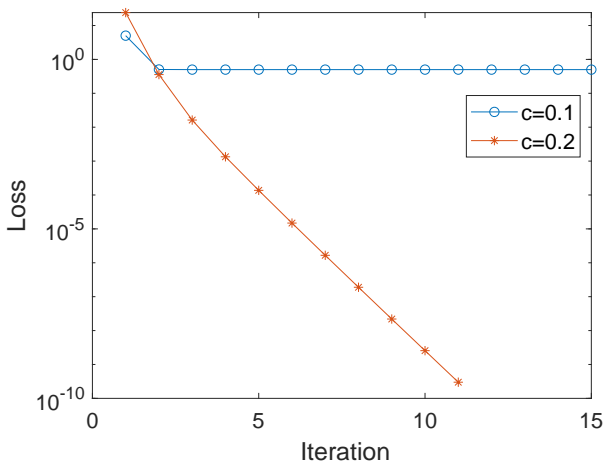}
    \caption{Typical convergence behavior of a ``trapping''
    and an ``escaping'' iteration with $c$ being $0.1$ and $0.2$
    respectively.} \label{fig:trap_escape}
\end{figure}

\begin{figure}[ht]
	\centering
    \includegraphics[height=0.5\textwidth]{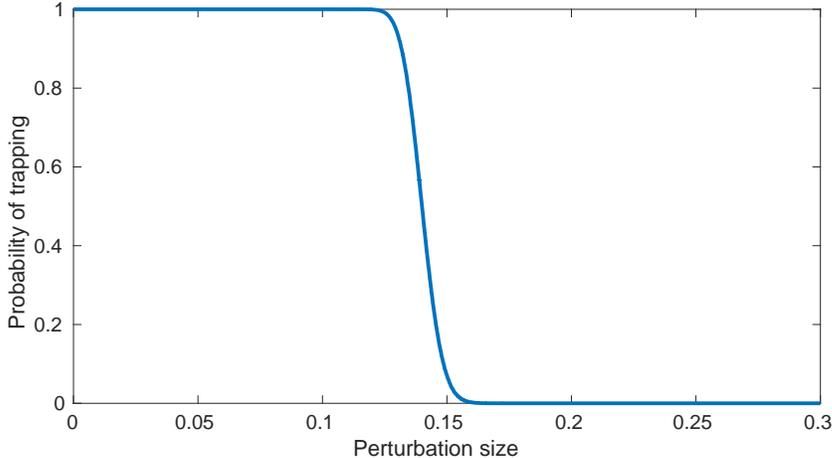}
    \caption{The stability of the spurious local minimum. The
    probability of trapping at $\vec{\mathbf{u}}_0$ via ALS against
    the perturbation size $c$.} \label{fig:locminphase}
\end{figure}

The curve in Figure~\ref{fig:locminphase} decreases monotonically,
which indicates that ALS is more likely to escape from the set
$\{\tu{}:\tau(\tu{})=\tau(\tu{}_0)\}$ if it starts at a point further
away from $\tu{}_0$. While, if the perturbation size is small enough,
\textit{i.e.}, ALS starts at a point close to local minimum $\tu{}_0$,
then ALS can not escape empirically. In Figure~\ref{fig:trap_escape},
we plot two curves of loss corresponding to $c=0.1$ and $c=0.2$
respectively. We can see that ALS traps when $c=0.1$ and escape when
$c=0.2$. This provides numerical evidence that the local minimum
$\tu{}_0$ is somewhat stable, \textit{i.e.}, $\tu{}_0$ is a interior
point of the basin of attraction of the ALS and the basin has
positive measure.

\subsection{One-loop convergence}

In this section, we numerically show that Theorem~\ref{thm:one-loop}
holds in practice and the bond dimension given in the theorem is
tight, \textit{i.e.}, one-loop convergence fails if $m=r^{d-1}-1$. We
present the results for different choices of $r\geq 3$, $d\geq 3$, and
$n\geq r^2$.  Two choices of bond dimensions, $m = r^{d-1}$ and $m =
r^{d-1}-1$, are tested. For a given $r$, $d$, $n$, and $m$, we randomly
generate an initial tensor ring with each entry being a standard
Gaussian random, and then perform ALS for $d$ microsteps. The final
function value $f(\tu{}_1)$ is reported as the result. Such experiment
is repeated for 100 times for every given $r$, $d$, $n$, and the
statistics of $f(\tu{}_1)$s are reported in Table~\ref{tab:one-loop}.

\begin{table}[htb]
    \centering
    \caption{One-loop convergence}
    \label{tab:one-loop}
    \begin{tabular}{ccccll}
        \toprule
        $d$ & $r$ & $n$ & $m$
        & $\max f(\tu{}_1)$ & $\min f(\tu{}_1)$ \\
        \toprule
        \multirow{2}*{3} & \multirow{2}*{3} & \multirow{2}*{10}
        & $9$ & $1.79 \times 10^{-13}$ & $1.01 \times 10^{-22}$ \\
        \cmidrule{4-6}
        & & & $8$ & $7.37 \times 10^{1}$ & $5.29 \times 10^{0}$ \\
        \midrule
        \multirow{2}*{4} & \multirow{2}*{3} & \multirow{2}*{10}
        & $27$ & $4.45 \times 10^{-9}$ & $6.69 \times 10^{-19}$ \\
        \cmidrule{4-6}
        & & & $26$ & $1.06 \times 10^{2}$ & $2.44 \times 10^{0}$ \\
        \midrule
        \multirow{2}*{3} & \multirow{2}*{4} & \multirow{2}*{16}
        & $16$ & $8.58 \times 10^{-5}$ & $5.10 \times 10^{-17}$ \\
        \cmidrule{4-6}
        & & & $15$ & $5.61 \times 10^{1}$ & $1.63 \times 10^{0}$ \\
        \bottomrule
    \end{tabular}
    
\end{table}

Table~\ref{tab:one-loop} shows that $\tu{}_1$ definitely does not
converge in one loop for bond dimension $m = r^{d-1} - 1$, since
$f(\tu{}_1)$ is far away from zero. Hence Theorem~\ref{thm:one-loop}
is numerically verified and so is its tightness. We notice that the
$\max f(\tu{}_1)$ is not close to machine accuracy after one loop
when $m = r^{d-1}$, which is due to the initial random tensor ring is
ill-conditioned and the numerical inverse is significantly polluted
by the numerical error.

\section{Conclusion}
\label{sec: conclusion}

In this paper, we investigate a sharp transition for the optimization
landscape associated with the tensor ring decomposition. Consider least
square fitting of a target tensor $\tT$ by $d$-th order TR format with
bond dimension $r^{d-1}$, or equivalently solving the optimization
problem \eqref{eq:trdobj} with bond dimension $r^{d-1}$, if $\tT$
is in TR format with bond dimension $r$, the problem is trivial,
\textit{i.e.}, one-loop convergence holds. However, if $\tT$ is in TR
format with bond dimension $r+1$, the landscape of \eqref{eq:trdobj}
might be quite bad, \textit{i.e.}, there may exist some spurious local
minima. These results tell us that even in the over-parameterized
case, we may not expect a good optimization landscape of tensor ring
decomposition, which in some sense shows the difficulty in numerical
algorithms for tensor ring decomposition.

\section*{Acknowledgments}

This work is partially supported by the National Science Foundation
under awards OAC-1450280 and DMS-1454939. Ziang Chen is partially
supported by the elite undergraduate training program of School of
Mathematical Sciences in Peking University.

\bibliographystyle{siam}
\bibliography{references}

\appendix

\newpage

\section{Visualization of the target tensor and the local minimum
in Theorem~\ref{thm:locmin}}
\label{visualize}

In the case that $d=3$ and $r=2$, the target tensor
$\tT_0\in\bbR^{5\times 5\times 5}$ is explicitly given by
\begin{equation*}
\begin{split}
    &
    \tT_0(:,:,1)=\begin{pmatrix}
    1 & 0 & 0 & 0 & 0\\
    0 & 0 & 0 & 0 & 0\\
    0 & 1 & 0 & 0 & 0\\
    0 & 0 & 0 & 0 & 0\\
    0 & 0 & 0 & 0 & 0
    \end{pmatrix}, \qquad\qquad
    \tT_0(:,:,2)=\begin{pmatrix}
    0 & 0 & 1 & 0 & 0\\
    0 & 0 & 0 & 0 & 0\\
    0 & 0 & 0 & 1 & 0\\
    0 & 0 & 0 & 0 & 0\\
    0 & 0 & 0 & 0 & 0
    \end{pmatrix}, \\
    &
    \tT_0(:,:,3)=\begin{pmatrix}
    0 & 0 & 0 & 0 & 0\\
    1 & 0 & 0 & 0 & 0\\
    0 & 0 & 0 & 0 & 0\\
    0 & 1 & 0 & 0 & 0\\
    0 & 0 & 0 & 0 & 0
    \end{pmatrix}, \qquad \qquad
    \tT_0(:,:,4)=\begin{pmatrix}
    0 & 0 & 0 & 0 & 0\\
    0 & 0 & 1 & 0 & 0\\
    0 & 0 & 0 & 0 & 0\\
    0 & 0 & 0 & 1 & 0\\
    0 & 0 & 0 & 0 & 0
    \end{pmatrix}, \\
    &
    \tT_0(:,:,5)=\begin{pmatrix}
    0 & 0 & 0 & 0 & 0\\
    0 & 0 & 0 & 0 & 0\\
    0 & 0 & 0 & 0 & 0\\
    0 & 0 & 0 & 0 & 0\\
    0 & 0 & 0 & 0 & 1
    \end{pmatrix}.
\end{split}
\end{equation*}
We emphasize that the last columns and rows of the first four slices are
all zeros. In the last slice, $\tT_0(:,:,5)$ is all zero except the last
entry is one, which comes from the extra term in \eqref{eq:T0}.

The local minimum $\tu{}_0=(\tu{1},\tu{2},\tu{3})$ defined in
\eqref{eq:u0} is,
\begin{equation} \label{eq:u0-1}
    \tu{1}_{\pi(p_1,p_2),\pi(q_1,q_2)} =
    \begin{cases}
        e_{\pi(p_1,q_1)},&\text{if } p_2 = q_2 \\
        0,&\text{otherwise}
    \end{cases},
\end{equation}
\begin{equation} \label{eq:u0-2}
    \tu{2}_{\pi(p_1,p_2),\pi(q_1,q_2)} =
    \begin{cases}
        e_{\pi(p_2,q_2)},&\text{if } p_1 = q_1\\
        0,&\text{otherwise}
    \end{cases},
\end{equation}
and
\begin{equation} \label{eq:u0-3}
    \tu{3}_{\pi(p_1,p_2),\pi(q_1,q_2)} =
    \begin{cases}
        e_{\pi(p_1,q_2)},&\text{if } p_2 = q_1\\
        0,&\text{otherwise}
    \end{cases}.
\end{equation}
In the setting $r=2$, $\tu{1}$, $\tu{2}$, and $\tu{3}$ are all $4\times
5\times 4$-tensors with
\begin{equation*}
\begin{split}
    &\tu{1}(:,1,:)=\begin{pmatrix}
    1 & 0 & 0 & 0\\
    0 & 1 & 0 & 0\\
    0 & 0 & 0 & 0\\
    0 & 0 & 0 & 0
    \end{pmatrix},
    \qquad \qquad
    \tu{1}(:,2,:)=\begin{pmatrix}
    0 & 0 & 1 & 0\\
    0 & 0 & 0 & 1\\
    0 & 0 & 0 & 0\\
    0 & 0 & 0 & 0
    \end{pmatrix}, \\
    &\tu{1}(:,3,:)=\begin{pmatrix}
    0 & 0 & 0 & 0\\
    0 & 0 & 0 & 0\\
    1 & 0 & 0 & 0\\
    0 & 1 & 0 & 0
    \end{pmatrix},
    \qquad \qquad
    \tu{1}(:,4,:)=\begin{pmatrix}
    0 & 0 & 0 & 0\\
    0 & 0 & 0 & 0\\
    0 & 0 & 1 & 0\\
    0 & 0 & 0 & 1
    \end{pmatrix},\\
    &\tu{1}(:,5,:)=\begin{pmatrix}
    0 & 0 & 0 & 0\\
    0 & 0 & 0 & 0\\
    0 & 0 & 0 & 0\\
    0 & 0 & 0 & 0
    \end{pmatrix};
\end{split}
\end{equation*}
\begin{equation*}
\begin{split}
    &\tu{2}(:,1,:)=\begin{pmatrix}
    1 & 0 & 0 & 0\\
    0 & 0 & 0 & 0\\
    0 & 0 & 1 & 0\\
    0 & 0 & 0 & 0
    \end{pmatrix},
    \qquad \qquad
    \tu{2}(:,2,:)=\begin{pmatrix}
    0 & 1 & 0 & 0\\
    0 & 0 & 0 & 0\\
    0 & 0 & 0 & 1\\
    0 & 0 & 0 & 0
    \end{pmatrix},\\
    &\tu{2}(:,3,:)=\begin{pmatrix}
    0 & 0 & 0 & 0\\
    1 & 0 & 0 & 0\\
    0 & 0 & 0 & 0\\
    0 & 0 & 1 & 0
    \end{pmatrix},
    \qquad \qquad
    \tu{2}(:,4,:)=\begin{pmatrix}
    0 & 0 & 0 & 0\\
    0 & 1 & 0 & 0\\
    0 & 0 & 0 & 0\\
    0 & 0 & 0 & 1
    \end{pmatrix},\\
    &\tu{2}(:,5,:)=\begin{pmatrix}
    0 & 0 & 0 & 0\\
    0 & 0 & 0 & 0\\
    0 & 0 & 0 & 0\\
    0 & 0 & 0 & 0
    \end{pmatrix};
\end{split}
\end{equation*}
and
\begin{equation*}
\begin{split}
    &\tu{3}(:,1,:)=\begin{pmatrix}
    1 & 0 & 0 & 0\\
    0 & 0 & 1 & 0\\
    0 & 0 & 0 & 0\\
    0 & 0 & 0 & 0
    \end{pmatrix},
    \qquad \qquad
    \tu{3}(:,2,:)=\begin{pmatrix}
    0 & 1 & 0 & 0\\
    0 & 0 & 0 & 1\\
    0 & 0 & 0 & 0\\
    0 & 0 & 0 & 0
    \end{pmatrix},\\
    &\tu{3}(:,3,:)=\begin{pmatrix}
    0 & 0 & 0 & 0\\
    0 & 0 & 0 & 0\\
    1 & 0 & 0 & 0\\
    0 & 0 & 1 & 0
    \end{pmatrix},
    \qquad \qquad
    \tu{3}(:,4,:)=\begin{pmatrix}
    0 & 0 & 0 & 0\\
    0 & 0 & 0 & 0\\
    0 & 1 & 0 & 0\\
    0 & 0 & 0 & 1
    \end{pmatrix},\\
    &\tu{3}(:,5,:)=\begin{pmatrix}
    0 & 0 & 0 & 0\\
    0 & 0 & 0 & 0\\
    0 & 0 & 0 & 0\\
    0 & 0 & 0 & 0
    \end{pmatrix}.
\end{split}
\end{equation*}
The last slide for any $\tu{k}$ are empty.

The tensor $\tau(\tu{}_0)\in\bbR^{5\times 5\times 5}$ is then the
same as $\tT_0$ except the last slice, i.e., $\tau(\tu{}_0)(:,:,5)$
is an all zero matrix.

\section{Examples of \mhp{}}
\label{app:mhp}

We first give an example for the entry-wise productivity property.
Let $f(x_1,x_2)$ and $g(x_1,x_3)$ be two \mhp{} with output being in
$\bbR^{2 \times 1}$, \textit{i.e.},
\begin{equation}
    f(x_1,x_2) =
    \begin{pmatrix}
        f_1(x_1, x_2) \\
        f_2(x_1, x_2)
    \end{pmatrix}
    \text{ and }
    g(x_1,x_3) =
    \begin{pmatrix}
        g_1(x_1, x_3) \\
        g_2(x_1, x_3)
    \end{pmatrix}.
\end{equation}
Then the entry-wise product of $f$ and $g$, denoted as
\begin{equation}
    h(x_1, x_2, x_3) =
    \begin{pmatrix}
        f_1(x_1, x_2) g_1(x_1, x_3) \\
        f_2(x_1, x_2) g_2(x_1, x_3)
    \end{pmatrix},
\end{equation}
is still a entry-wise polynomial of entries of $x_1, x_2, x_3$.
The homogeneity can be justified as,
\begin{equation}
    \begin{split}
        h(\lambda x_1, x_2, x_3) = & \lambda^{s_1^f + s_1^g} h(x_1, x_2,
        x_3), \\
        h(x_1, \lambda x_2, x_3) = & \lambda^{s_2^f} h(x_1, x_2,
        x_3), \text{ and} \\
        h(x_1, x_2, \lambda x_3) = & \lambda^{s_3^g} h(x_1, x_2,
        x_3), \\
    \end{split}
\end{equation}
where $s_1^f$, $s_2^f$, $s_1^g$, and $s_3^g$ are the homogeneity
degree of $f$ and $g$ respectively.  Hence $h$ is also \mhp{}.

We then give an example for the tensor contraction (matrix product)
productivity property.  Let $f(x_1,x_2)$ and $g(x_1,x_3)$ be two \mhp{}
with output being in $\bbR^{2 \times 2}$, \textit{i.e.},
\begin{equation}
    \begin{split}
        f(x_1,x_2) = &
        \begin{pmatrix}
            f_{11}(x_1, x_2) & f_{12}(x_1, x_2) \\
            f_{21}(x_1, x_2) & f_{22}(x_1, x_2)
        \end{pmatrix}
        \text{ and } \\
        g(x_1,x_3) = &
        \begin{pmatrix}
            g_{11}(x_1, x_3) & g_{12}(x_1, x_3) \\
            g_{21}(x_1, x_3) & g_{22}(x_1, x_3)
        \end{pmatrix}.
    \end{split}
\end{equation}
The matrix product of $f$ and $g$, denoted as
\begin{equation}
    h(x_1, x_2, x_3) =
    \begin{pmatrix}
        f_{11} g_{11} + f_{12} g_{21} & f_{11} g_{12} + f_{12} g_{22} \\
        f_{21} g_{11} + f_{22} g_{21} & f_{21} g_{12} + f_{22} g_{22} \\
    \end{pmatrix},
\end{equation}
is still a entry-wise polynomial of entries of $x_1, x_2, x_3$.
The homogeneity can be justified as,
\begin{equation}
    \begin{split}
        h(\lambda x_1, x_2, x_3) = & \lambda^{s_1^f + s_1^g} h(x_1, x_2,
        x_3), \\
        h(x_1, \lambda x_2, x_3) = & \lambda^{s_2^f} h(x_1, x_2,
        x_3), \text{ and} \\
        h(x_1, x_2, \lambda x_3) = & \lambda^{s_3^g} h(x_1, x_2,
        x_3), \\
    \end{split}
\end{equation}
where $s_1^f$, $s_2^f$, $s_1^g$, and $s_3^g$ are the homogeneity
degree of $f$ and $g$ respectively.  Hence $h$ is also \mhp{}.

Finally, we give an example for the composition property.  Let
$f(x_1,x_2,y)$ and $g(x_1,x_3)$ be two \mhp{} and $g$ is of the same
dimension as $y$. The composition of $f$ and $g$, denoted as
\begin{equation}
    h(x_1, x_2, x_3) = f(x_1, x_2, g(x_1, x_3)),
\end{equation}
is still a entry-wise polynomial of entries of $x_1, x_2, x_3$.
The homogeneity can be justified as,
\begin{equation}
    \begin{split}
        h(\lambda x_1, x_2, x_3) = & \lambda^{s_1^f + s_y^f s_1^g}
        h(x_1, x_2, x_3), \\
        h(x_1, \lambda x_2, x_3) = & \lambda^{s_2^f} h(x_1, x_2,
        x_3), \text{ and} \\
        h(x_1, x_2, \lambda x_3) = & \lambda^{s_y^f s_3^g} h(x_1, x_2,
        x_3), \\
    \end{split}
\end{equation}
where $s_1^f$, $s_2^f$, $s_y^f$, $s_1^g$, and $s_3^g$ are the
homogeneity degree of $f$ and $g$ respectively.  Hence $h$ is also
\mhp{}.

\end{document}